\documentclass[reqno,centertags, 11pt]{amsart}
\usepackage{comment,graphicx,url}
\usepackage{color}

\definecolor{purple}{rgb}{0.65, 0, 0.9}
\definecolor{orange}{rgb}{1,.5,0}
\definecolor{gray}{rgb}{0.7,.7,0.7}

\usepackage{enumitem}
\usepackage[]{hyperref}
\usepackage{mathtools}
\usepackage{esint}

\newcommand{\quotes}[1]{``#1''}
\usepackage{graphics}
\usepackage{amssymb,amsmath,amsfonts}
-\marginparwidth \oddsidemargin -9mm \evensidemargin \oddsidemargin
\usepackage{latexsym}
\advance\hoffset by 5mm

\def\@abssec#1{\vspace{.1in}\footnotesize \parindent .2in
{\bf #1. }\ignorespaces}
\newtheorem{theorem}{Theorem}[section]

\newtheorem{lemma}[theorem]{Lemma}
\newtheorem{proposition}[theorem]{Proposition}

\newtheorem{definition}[theorem]{Definition}
\newtheorem{remark}[theorem]{Remark}

\newcommand*{\rom}[1]{\expandafter\@slowromancap\romannumeral #1@}

\newcommand{\be}{\mathbf e}

\allowdisplaybreaks \numberwithin{equation}{section}

\renewcommand{\be}{\begin{equation}}
\newcommand{\ee}{\end{equation}}

\newcommand{\colvec}[1]{\begin{pmatrix} #1 \end{pmatrix}}

\begin{document}

\title[Boussinesq]{Long time stability and instability in the two-dimensional Boussinesq system with kinematic viscosity}

\author{Jaemin Park}
\address{Department Mathematik Und Informatik, Universit\"at Basel, CH-4051 Basel, Switzerland}
\email{jaemin.park@unibas.ch}

\thanks{\textit{Acknowledgements}. The author acknowledges  the support of the SNSF Ambizione grant No. 216083. }
\begin{abstract}
 In this paper, we investigate the long-time behavior of the two-dimensional incompressible Boussinesq system with kinematic viscosity in a periodic channel, focusing on  instability and asymptotic stability near hydrostatic equilibria. Firstly, we prove that any hydrostatic equilibrium reveals long-time instability when the initial data are perturbed in  Sobolev spaces of low regularity. Secondly, we establish  asymptotic stability of the stratified density, which is strictly decreasing in the vertical direction, under sufficiently regular perturbations, proving that the solution converges to the unique minimizer of the total energy.
 
  Our analysis is based on the energy method. Although the total energy dissipates due to kinematic viscosity, such mechanism cannot capture the stratification of the density. We overcome this difficulty by discovering another  Lyapunov functional which exhibits the density stratification in a quantitative manner. 
\end{abstract}

\maketitle
\setcounter{tocdepth}{1}

\tableofcontents

\section{Introduction}
 Our aim in this paper is to study long-time behavior of the  two-dimensional Boussinesq system in a periodic channel, $\Omega:=\mathbb{T}\times (0,1)$, 
\begin{align}\label{Bouss}
\begin{cases}
\rho_t + u\cdot\nabla \rho = 0,\\
u_t + u\cdot\nabla u = -\nabla p - g\colvec{0 \\ \rho}+\nu \Delta u, & \text{ $t\ge 0$, $x=(x_1,x_2)\in \Omega$}\\
\nabla\cdot u =0,
\end{cases}
\end{align}
combined with the boundary and initial condition,
\begin{align}\label{bdcondition}
u=0 \text{ on $\partial \Omega$},\quad 
(\rho(0,x),u(0,x))=(\rho_0(x),u_0(x)).
\end{align}
Here, $\rho$ represents a scalar-valued function denoting the temperature density of fluids, while $u$ represents the velocity field that drives the fluids. The positive constants $g$ and $\nu$ denote the gravitational acceleration and  the kinematic viscosity coefficient, respectively. For simplicity we will assume that
\[
g= \nu = 1.
\]
The system \eqref{Bouss}  describes the evolution of the temperature distribution of a viscous, heat-conducting fluid under the effect of gravitational force (see, for example, \cite{blandford2008applications,gill1982atmosphere,majda2003introduction,pedlosky1987geophysical} for physical derivations of the model).  

\subsection{A brief overview of background}
  In regard to long-time behavior of the Boussinesq system, there are numerous results available in the literature. Perhaps the most necessary result in our theme is a global well-posedness theory to ensure the existence of solutions for large time. Considering \eqref{Bouss} in a general bounded domain with smooth boundary, it was proved in \cite{MR3135451} that if $(\rho_0,u_0)\in H^{k-1}\times H^{k}$ for $k\ge 2$, there exists a global in time solution $(\rho,u)\in L^{\infty}_{loc}([0,\infty); H^{k-1}\times H^{k})$ (to be more precise, a detailed proof in \cite{MR3135451} is provided only for the case $k=2$. However, as noted in their paper, the same method can be generalized to higher regularity class by adapting standard techniques to deal with the viscosity effect on the boundary (see e.g.,  \cite[Chapter 5]{galdi2000introduction} for such methods used in the well-posedness analysis of the Navier--Stokes equations in a bounded domain). When  the Boussinesq system is considered in the spatial domain $\mathbb{R}^2$, the global existence of solutions in $(\rho,u)\in L^{\infty}_{loc}([0,\infty); H^{k}\times H^{k})$ for $k\ge 3$ is available in a more explicit statement \cite{chae2006global,hou2005global}.

   We are particularly interested in asymptotic stability and instability of a class of steady states known as \textit{hydrostatic} equilibria:
\begin{align}\label{hydrostatic}
(\rho_s(x),u_s(x))=(\rho_s(x_2),0).
\end{align}
That is, the density $\rho_s$ is stratified in the sense that it is independent of the horizontal variable $x_1$, and the fluids are at rest. It is   straightforward to see that any solution of the form  \eqref{hydrostatic} is a stationary solution to the Boussinesq system. Moreover, the converse  also holds true: any steady state of the system \eqref{Bouss} must be a hydrostatic equilibrium. To see this, we consider the potential, kinetic and total energies:
\[
E_P(\rho):= \int_{\Omega}\rho(x)x_2dx,\quad E_K(u):=\frac{1}{2}\int_{\Omega}|u(x)|^2dx,\quad E_T(\rho,u):=E_P(\rho) + E_K(u).
\]
In  the system \eqref{Bouss}, the loss of total energy is solely attributed to the velocity dissipation. Denoting by $(\rho(t),u(t))$ a classical solution, a standard energy estimate reads (see Section~\ref{Lyapunsec} for detailed computations)
\begin{align}\label{energy_balance}
\frac{d}{dt}E_T(\rho(t),u(t)) = -\rVert \nabla u(t)\rVert_{L^2}^2.
\end{align}
Thus, in order for $(\rho_s,u_s)$ to be a steady state, the energy dissipation must be zero, yielding  $\nabla u_s=0$. Due to the non-slip boundary condition \eqref{bdcondition},  it must hold that $u_s\equiv 0$. Hence the velocity equation in \eqref{Bouss} tells us the external force field, $(0,\rho_s)^{T}$, must be a gradient field. Consequently  $\rho_s$ is stratified.  We note that the above argument might not be justified when the system is considered in an unbounded domain. Allowing the solution to have infinite (total) energy or the fluids to slip on the boundary, one can mathematically derive non-hydrostatic equilibria such as shear flows (see, for instance, \cite{masmoudi2022stability,zillinger2021enhanced}).
   
  Concerning the stability analysis near hydrostatic equilibria, most of the available research has been focused on a particular density profile:
  \begin{align}\label{linear_dens}
  \rho_s(x_2) = 1-\alpha x_2,\text{ for some $\alpha>0$}.
  \end{align}
  since such a  linear density profile reduces much of the technical difficulty in the analysis of the linearized system. Also, the sign-condition on the coefficient $\alpha>0$ is clearly necessary to establish stability; otherwise instability is observed even in the linearized system \cite[Theorem 1.4 (3)]{MR3815212} (such instability, resulting from the density gradient, is often referred to as \textit{Rayleigh--B\'enard instability}).  As an extension of the analysis on the linearized system conducted in \cite{MR3815212}, the authors in \cite{tao2020stability} investigated nonlinear stability in the spatial domain $\mathbb{T}^2$, proving that  the density  $\rho(t)$ will eventually converge to a stratified one as $t\to \infty$, under the additional assumption that $\sup_{t>0}\sum_{k\in \mathbb{Z}^2}|\hat{\rho}(k_1,k_2,t)| <\infty,$
  where $\hat{\rho}$ denotes the usual Fourier coefficient of $\rho$. 
  
   There are several results concerning the stability of slightly different models. The authors in \cite{MR3974167} replaced $\Delta u$ by $-u$ in the velocity equation of \eqref{Bouss}, and established asymptotic stability of hydrostatic equilibria with the linear density as in \eqref{linear_dens}. We also mention the work \cite{dong2022asymptotic}, where the authors established asymptotic stability of the system \eqref{Bouss} but with a different boundary condition which prevents the creation of vorticity on the boundary. We will give a  slightly more detailed comparison between these models and ours in Remark~\ref{comparison1}.

    \subsection{Some observations for the energy structure}\label{obsermay}
    A trivial observation from \eqref{energy_balance} is that the total energy is always decreasing. Integrating \eqref{energy_balance} over time, we obtain
   \begin{align}\label{integrat222ingint}
 \frac{1}{t}\int_0^t \rVert \nabla u(s)\rVert_{L^2}^2ds =   \frac{1}{t}\left(E_T(\rho_0,u_0)-E_{T}(\rho(t),u(t))\right) = O(t^{-1}),\text{ for all $t\gg 1$}
   \end{align}
 As $t\to \infty$, we observe $\rVert \nabla u(s)\rVert_{L^2}$ decays like $O(t^{-1/2})$, in an average sense. However the velocity converging to $0$  is not enough to conclude that the solution  converges to a steady state. The  density may induce velocity again, unless it is exactly stratified. Indeed, the main difficulty is that the energy balance \eqref{energy_balance}  neither captures the energy transfer between the potential and kinetic energies nor provides any information about the density stratification.
  
   In order to overcome this difficulty, we introduce an auxiliary vector field $v=v(t)$, a solution to the so-called steady Stokes equation:
\begin{align}\label{sB22vec}
\begin{cases}
-\Delta v = -\nabla q -\colvec{0 \\ \rho(t)},\quad \nabla\cdot v=0, & \text{ in  $\Omega$}\\
v = 0. & \text{ on $\partial\Omega$.}
\end{cases}
\end{align} 
The motivation behind this definition is as follows:  In the velocity equation in \eqref{Bouss}, we formally neglect the terms $u_t$ and $u\cdot \nabla u$, since they are expected to become much smaller, compared to the other terms, as the solution approaches to a hydrostatic equilibrium. The resulting velocity equation coincides with \eqref{sB22vec}. In other words, we expect $u-v$ to vanish throughout the evolution.  To quantify this expectation, we consider a quantity $S$ defined by
\[
S(t):=\int_{\Omega} |u(t)-v(t)|^2 dx.
\]
 Our key observation is that a suitable linear combination of $S(t)$ and the total energy is a Lyapunov functional (see Proposition~\ref{lyaypiunove} for a precise statement):
 \[
 \frac{d}{dt}\left(E_T(\rho,u)+ S \right)\lesssim -(\rVert \nabla v\rVert_{L^2}^2 + \rVert \nabla u\rVert_{L^2}^2 +\rVert \nabla(u-v)\rVert_{L^2}^2 ).
 \]
 Integrating this over time, similar computations as in \eqref{integrat222ingint} now give us
 \[
 \frac{1}{t}\int_0^t \rVert \nabla v(s)\rVert_{L^2}^2 + \rVert \nabla u(s)\rVert_{L^2}^2 + \rVert \nabla(u-v)\rVert_{L^2}^2ds = O(t^{-1}),\text{ for $t\gg 1$.}
 \]
This reveals that the kinetic energies stored in both vector fields $u$ and $v$  decay over time in a time average sense,
 \begin{align}\label{quat_r}
  \rVert  \nabla u\rVert_{L^2} + \rVert  \nabla v\rVert_{L^2}=O(t^{-1/2}).
 \end{align}
 We emphasize that the new Lyapunov functional $E_T+S$ is certainly more useful than the total energy in that it provides  quantitative information for the density stratification. The  Biot-Savart law for $v$ in \eqref{sB22vec} corresponds to $v=\nabla^\perp \Delta_{\Omega}^{-2}\partial_1\rho(t)$. Hence the decay rate \eqref{quat_r} formally yields
 \begin{align}\label{anisot1}
\rVert \partial_1\rho\rVert_{H^{-2}} = O(t^{-1/2}),
 \end{align}
 indicating that the dependence on $x_1$ of $\rho$ vanishes in a weak sense as $t\to \infty$. 
 
\subsection{Main results for instability and the asymptotic stability}
  As a first application of this energy structure, we present a  long-time instability of arbitrary  equilibrium under small perturbations in $H^{2-}$. This is motivated by the instability mechanism introduced in \cite[Theorem 1.5]{kiselev2023small} for the incompressible porous media equation.  
    \begin{theorem}\label{instability}
Let $\rho_s$ be a smooth stratified density, that is, $\rho_s(x)=\rho_s(x_1,x_2)$ is independent of $x_1$. For any $\gamma,\epsilon>0$ and initial velocity $u_0\in C^\infty(\Omega)$ such that $\nabla\cdot u_0=0$ and $u_0=0$ on $\partial\Omega$, there exists an initial density $\rho_0\in H^{2-\gamma}\cap C^\infty(\Omega)$ such that  
\[
\rVert \rho_0 -{\rho}_s\rVert_{H^{2-\gamma}(\Omega)}\le \epsilon
\]
and 
\[
 \limsup_{t\to\infty} t^{-k/4}\rVert\partial_1\rho(t)\rVert_{H^{k}(\Omega)} =\infty,\text{ for all $k> 0$.}
\]
\end{theorem}
\begin{remark}
The proof of Theorem~\ref{instability} heavily relies on the argument made in \cite[Section 4-5]{kiselev2023small} in the context of the incompressible porous media equation (IPM), where the authors connected the decay of an anisotropic negative Sobolev norm to a growth of high Sobolev norms. We also mention that the same idea was  adapted to prove instability in the Stokes transport system  (Stokes) in \cite[Theorem 3.7.2]{leblond2023well}. The only difference in our setting, compared to these models, is that the potential energy is not monotone decreasing.  Our theorem will be established by rigorously justifying the estimate given in  \eqref{anisot1}.
\end{remark}

 Our second main result of this paper is to establish asymptotic stability near stable hydrostatic equilibria.  In this stability analysis, we restrict the class of steady states to prevent the Rayleigh--B\'enard instability. More precisely, we will consider $\rho_s$ such that
 \begin{align}\label{steadtysd2sd}
\gamma:=\inf_{x_2\in[0,1]}\left(-\partial_2{\rho}_s(x_2)\right)>0,\text{ and } \partial_2{\rho}_s\in H^4([0,1]). 
\end{align}
 Let us consider a solution $(\rho(t),u(t))$ whose initial data is a small perturbation from a hydrostatic equilibrium. A trivial but crucial property of the Boussinesq system~\eqref{Bouss} is that the density is transported by an incompressible flow, which preserve the area of each super level set.  If $\rho(t)$ indeed asymptotically converges to another stratified density in a sufficiently strong topology (strong enough to preserve topological properties of the super level sets), the limit density must be given by a push-forward of a measure-preserving diffeomorphism of $\rho_0$. This motivates us to define a vertical (decreasing) rearrangement as 
   \begin{align}\label{verticla_re}
 \rho\mapsto  \rho^*(x):=\int_0^\infty 1_{\left\{0\le x_2 \le |\left\{ \rho > s\right\}|\right\}}ds.
   \end{align}
 From the above consideration, it should come as no surprise  that the vertical rearrangement of the initial density, $\rho_0^*$, is the unique candidate for the long-time limit of  the density, $\lim_{t\to\infty}\rho(t)$.

 Another important property of the vertical rearrangement is  that $\rho^*$ is the unique minimizer of the potential energy, among all the densities which can be obtaind by a push-forward of a measure-preserving diffeomorphism of $\rho$. For a solution $(\rho(t),u(t))$ to the Boussinesq system, this observation implies   $ E_P(\rho(t))\ge E_P(\rho_0^*)$, resulting in
 \[
E_P(\rho(t)) + E_K(u(t)) \ge E_P(\rho_0^*),\text{ for all $t>0$}.
\]  
 Hence, the asymptotic stability in the Boussinesq system can be interpreted as the convergence in the solution curve $t\to (\rho(t),u(t))$, which minimizes the total energy. The precise statement  is as follows:
\begin{theorem}\label{main_Boussinesq}
Let $(\rho_s,0)$ be a hydrostatic equilibrium satisfying \eqref{steadtysd2sd}. There exist $\epsilon=\epsilon(\gamma,\rVert \rho_s\rVert_{H^4(\Omega)})$ and $C=C(\gamma,\rVert \rho_s\rVert_{H^4})$ such that 
if $(\rho_0-{\rho}_s,u_0)\in \left(H^2_0(\Omega)\cap H^4(\Omega)\right)^2$ and 
\begin{align*}
\rVert \rho_0-{\rho}_s\rVert_{H^4}+\rVert u_0\rVert_{H^4}\le \epsilon,
\end{align*} then the  solution $(\rho(t),u(t))$ to the Boussinesq system \eqref{Bouss} satisfies  
\begin{align}\label{mnum12}
\rVert \rho(t)-{\rho}_s\rVert_{H^4(\Omega)}  + \rVert u(t)\rVert_{H^4}\le C\epsilon,\text{ for all $t>0$}.
\end{align} Furthermore, the total energy decays as 
\begin{align*}
E_P(\rho(t))-E_P(\rho_0^*) + E_K(u(t))\le C\epsilon^2 t^{-2},
\end{align*}
where $\rho_0^*$ is the vertical rearrangement of the initial density. Consequently, the solution $(\rho(t),u(t))$ converges to  $(\rho_0^*,0)$, satisfying
\begin{align}\label{mnum122}
\rVert \rho(t)-\rho_0^*\rVert_{L^2(\Omega)} + \rVert u(t)\rVert_{L^2}\le C\frac{\epsilon}{t},\text{ for all $t>0$.}
\end{align}
\end{theorem}

\begin{remark}\label{comparison1}
In the aforementioned works \cite{MR3974167,dong2022asymptotic},  asymptotic stability near hydrostatic equilibrium with $\rho_s(x)=1-x_2$ was established for slight variations of our Boussinesq system. In \cite{MR3974167}, the authors considered a model where the viscous damping $\Delta u$ is replaced by $-u$ in the velocity equation in \eqref{Bouss}, establishing asymptotic stability under perturbations in $H^{k}$ for $k\ge 16$. We note that our damping term $\Delta u$ cannot be viewed  stronger than $-u$, simply  counting the number of derivatives. Although our model induces stronger damping effect on the velocity, it induces strictly weaker damping effect on the density, which can be  detected from the linearized system. In \cite{dong2022asymptotic}, the authors modified the boundary condition \eqref{bdcondition} so that the flow slips and the vorticity always vanishes on the boundary,  preventing the boundary layer formation. In this setting, asymptotic stability was established under perturbations in $H^{k}$ for $k> 32$. The main difference of Theorem~\ref{main_Boussinesq}, compared to these works, is that our model reveals nontrivial difficulty due to the viscosity effect on the boundary. Another difference is that our method does not rely on the analysis of the linearized system, allowing us to achieve  asymptotic stability in much weaker regularity class of perturbations.
\end{remark}

\begin{remark}
The decay rate of the total energy $E_T(t)=O(t^{-2})$ obtained in Theorem~\ref{main_Boussinesq} matches with the numerical experiment in \cite{MR3815212} (note that the the initial data considered in their paper are not close to a hydrostatic equilibrium, and the simulation was performed in a spatial domain $[0,1]^2$ instead of the periodic channel). Moreover, our proof exhibits that the rate of the density stratification in \eqref{anisot1} can be improved. Indeed we will obtain
\[
\frac{2}{t}\int_{t/2}^t \rVert \nabla v(s)\rVert_{L^2}^2 + \rVert \nabla u(s)\rVert_{L^2}^2 ds \le_C \epsilon^2t^{-3},\text{ for all $t>0$},
\]
thus the long-time averages of $\rVert \partial_1\rho (t)\rVert_{H^{-2}}$ and $ \rVert \nabla u(t)\rVert_{L^2}$ decay like $O(t^{-3/2})$ at least.  We note that such decay rates are not directly deduced from interpolating \eqref{mnum12} and \eqref{mnum122}.
\end{remark}

Although the proof of Theorem~\ref{main_Boussinesq} is slightly lengthy, the key idea is simple. Adapting our observations described in Subsection~\ref{obsermay}, we treat the velocity difference $u-v$ as a perturbation. In the extreme case where the two velocities $u$ and $v$ are exactly equal, the transport equation for the density can be read as
\begin{align}\label{rkasd2sd}
\rho_t+ v\cdot \nabla\rho =0,\quad \text{$v$ is given by \eqref{sB22vec}.}
\end{align}
This is an active scalar equation, often referred to as   {the Stokes  transport system} (see, for example, \cite{mecherbet2018sedimentation} for a physical motivation of this model). Although the damping effect on the density in \eqref{rkasd2sd} may not be immediately apparent, an inviscid damping mechanism of this system was established in the work \cite{dalibard2023long}, motivated by the stability analysis for the  IPM equation \cite{elgindi2017asymptotic} (here, to avoid any potential confusion regarding the terminology, we note that  the term \quotes{inviscid damping} is used because the density in \eqref{rkasd2sd} does not exhibit any viscosity-induced dissipation, although the flow determined by \eqref{sB22vec} is viscous). We adapt this inviscid damping mechanism  based on  the proof in \cite{park2024stability} for the Stokes transport system. Meanwhile the perturbative quantity, $u-v$, will be  effectively estimated throughout the proof. 

\subsection{Organization of the paper} In Section~\ref{prelm}, we review classical lemmas regarding divergence free vector fields. In Section~\ref{Lyapunsec}, we give more detailed derivation of the new Lyapunov functional. The instability and the stability will be proved in Section~\ref{ins} and Section~\ref{assssss}, respectively. In the proof of asymptotic stability, several  estimates coming from some elementary ODEs will be used. Such estimates will be collected in Appendix~\ref{ode11sx}.

\subsection{Notations}
 Following the conventional practice, we denote by $C$ an implicit positive constant that may vary from line to line. In the case where $C$ depends on a quantity,  say $A$, we will represent it  as $C_A$ or $C(A)$. For two quantities, $A$ and $B$, we will also use the notation $A \le_C B$, indicating that $A \le CB$ for some constant $C > 0$.

\section{Preliminaries}\label{prelm}
\subsection{The steady Stokes equation and the Leray projection}
We  recall classical results concerning the steady Stokes equation and the Leray projection. Although the following lemmas hold true in general bounded domains, we will always assume that $\Omega=\mathbb{T}\times (0,1)$ is a periodic channel in our paper.

Given a vector field  $f$ in $\Omega$, let us consider the steady Stokes equation for $(v,q)$:
\begin{align}\label{steady_st}
\begin{cases}
-\Delta v = - \nabla q - f ,\quad \nabla \cdot v=0&\text{ in $\Omega$},\\
v=0 & \text{ on $\partial\Omega$},\\
\int_{\Omega}q(x)dx =0.
\end{cases}
\end{align}

\begin{lemma}\cite[ Theorem IV 6.1]{galdi2011introduction}\label{Stokes_wellposed}
Let  $f\in H^k(\Omega)$ for $k\ge -1$. The system \eqref{steady_st} has a unique  solution $(v,q)$ satisfying
\[
\rVert v\rVert_{H^{k+2}} +\rVert q\rVert_{H^{k+1}}\le C_k\rVert f\rVert_{H^k}.
\]
\end{lemma}

 Let $(v,q)$ be a solution to the steady Stokes equation \eqref{steady_st}. Thanks to the divergence free condition, the velocity $v$ can be represented in terms of a stream function, $v=\nabla^\perp \Psi$ for some $\Psi:\Omega\mapsto \mathbb{R}$. In order to characterize the the stream function, we recall the following lemma concerning the bilaplacian operator:
    \begin{lemma}\cite[Lemma B.1]{dalibard2023long}\label{lemma_psi_stokes}
    Let $h\in H^k(\Omega)$ for $k\ge -2$. Then the bilaplacian equation,
    \[
    \begin{cases}
    \Delta^2 \Psi =  h & \text{ in $\Omega$},\\
    \Psi =\nabla\Psi = 0 & \text{ on $\partial \Omega$,}
    \end{cases}
    \]
    admits a unique solution $\Psi\in H_{0}^2(\Omega)\cap H^{k+4}(\Omega)$ and it satisfies
    \[
    \rVert \Psi\rVert_{H^{k+4}}\le C_k \rVert h\rVert_{H^k}.
    \]
    \end{lemma}
Combining the above two lemmas, we obtain the following:
\begin{lemma}\label{rigorousnavier}
Let $f\in H^k$ for $k\ge -1$. Let $(v,q)$ be the unique solution to \eqref{steady_st}. Then 
\[
v=\nabla^\perp\Psi:=(-\partial_2\Psi,\partial_1\Psi)^T,
\] where $\Psi$ is the unique solution to
 \[
 \begin{cases}
 \Delta^{2}\Psi =\partial_1f_2-\partial_2f_1 & \text{ in $\Omega$},\\
 \Psi = \nabla\Psi = 0, &\text{ on $\partial\Omega$.}
 \end{cases}
 \]    
\end{lemma}    
 \begin{proof}
It is straightforward to see that 
 \[
\nabla\times \Delta (v-\nabla^\perp \Psi)=\nabla \times (f - \nabla^\perp \Delta\Psi) = \nabla \times f - \nabla\times \nabla^\perp\Delta \Psi=0.
 \]
 Since $\Omega$ is simply connected, the vector field $\Delta(v-\nabla^\perp \Psi)$ must be a gradient field. Therefore  there exists $p:\Omega\mapsto\mathbb{R}$ such that $(v-\nabla^\perp \Psi,p)$ solves
 \[
 \begin{cases}
 -\Delta (v-\nabla^\perp\Psi) =  -\nabla p,\quad \nabla\cdot (v-\nabla^\perp \Psi)=0, & \text{ in $\Omega$},\\
    v-\nabla^\perp\Psi = 0 & \text{ on $\partial \Omega$,}\\
    \int_{\Omega}p(x)dx=0.
    \end{cases}
 \]
Applying Lemma~\ref{Stokes_wellposed}, we conclude that $(v-\nabla^\perp \Psi,p)=(0,0)$ is the unique solution. In other words, $v=\nabla^\perp \Psi$.     
    \end{proof}
    
    Another simple application of Lemma~\ref{lemma_psi_stokes} is given below:
 \begin{lemma}\label{bilap}
For $f\in H^4(\Omega)$ such that  $f=\partial_2f =0$ on $\partial\Omega$ and $\int_{\Omega}f(x)dx=0$, it holds that
\[
\rVert f\rVert_{{H}^2(\Omega)} \le_C \rVert \Delta f\rVert_{L^2(\Omega)},\quad \rVert f\rVert_{H^4(\Omega)}\le_C \rVert \Delta^2f\rVert_{L^2(\Omega)}.
\]
\end{lemma}
\begin{proof}
The first estimate concerning the Laplacian operator follows trivially from integration by parts and the Poincar\'e inequality. For the second estimate, it is tautological that $f$ solves 
\[
\Delta^2f = \Delta^2 f \text{ in $\Omega$ and }f=\nabla f=0, \text{ on $\partial\Omega$.}
\]
Hence, Lemma~\ref{lemma_psi_stokes} gives $\rVert f\rVert_{H^4}\le_C \rVert \Delta^2 f\rVert_{L^2}$.
\end{proof}
 
Lastly, we  recall the Leray projection in a bounded domain.
\begin{lemma}\label{leray_pro}\cite[Proposition 1.8]{constantin1988navier}
Given a vector field $f\in L^2(\Omega)$, the differential equation for $(v,q)$ given by
\[
\begin{cases}
v=-\nabla q - f,\quad \nabla\cdot  v=0, & \text{ in $\Omega$}\\
v\cdot \vec{n} = 0,& \text{ on $\partial\Omega$}\\
\int_{\Omega}q(x)dx =0,
\end{cases}
\]
has a unique  solution $(v,q)$ satisfying
\[
\rVert v\rVert_{L^2} +\rVert q\rVert_{H^1}\le \rVert f\rVert_{L^2}.
\]
Given a vector field $f$, we denote  the solution $v$ by 
\begin{align}\label{leraty_pro}
\mathbb{P}f := v.
\end{align}
\end{lemma}
\section{Lyapunov functional in the Boussinesq system}\label{Lyapunsec}
In the rest of the paper, we denote $(\rho,u)$, solution to the Boussinesq system \eqref{Bouss}. Given the solution, let us define an auxiliary vector field $v$ as
\begin{align}\label{sBvec}
\begin{cases}
-\Delta v = -\nabla q -\colvec{0 \\ \rho}, \quad \nabla \cdot v=0 & \text{ in $\Omega$,}\\
v = 0 & \text{ on $\partial\Omega$.}
\end{cases}
\end{align} 
Note that the existence of such a vector field $v$ (also the pressure $q$) is ensured by Lemma~\ref{Stokes_wellposed}.  
Differentiating \eqref{sBvec} in time, we observe that $v_t$ satisfies
\begin{align}\label{sBvec23s22}
-\Delta v_t = -\nabla q_t -\colvec{0 \\ -\nabla\cdot (u \rho)},\quad 
\nabla\cdot v_t=0, \text{ in $\Omega$ and },
v_t = 0 \text{ on $\partial\Omega$.}
\end{align} Denoting
\begin{align}\label{wdef_1db}
w:= u -v, \quad F:=u\cdot\nabla u + v_t,
\end{align}
we see from the velocity equation in \eqref{Bouss}  that
\[
\Delta w = \Delta u - \Delta v = \nabla p + u\cdot\nabla u +u_t = \nabla p + u\cdot\nabla u + v_t + w_t = \nabla p + F + w_t.
\]
Combining this with the boundary condition and the incompressibility condition for $w$, we see that $w$ solves
\begin{align}\label{evol_w}
\begin{cases}
w_t  = - \nabla p - F +\Delta w, \quad \nabla \cdot w =0,  & \text{ in $\Omega$}\\
w = 0 & \text{ on $\partial \Omega$}.
\end{cases}
\end{align}
We define the following energy quantities:
 \[
 E_P(t):=\int (\rho(t)-\rho_0^*) x_2 dx,\quad E_K(t):=\frac{1}{2}\int |u(t)|^2dx,\quad E_T(t):= E_P(t)+ E_K(t),
  \]
  and 
  \[
  S(t):=\frac{1}{2}\int |w(t)|^2dx.
  \]
 It follows straightforwardly from \eqref{Bouss} that
\begin{align}
\frac{d}{dt}E_P(t) &= \int \rho_t x_2 dx=- \int u\cdot \rho x_2 dx= \int u_2 \rho dx,\label{Bou_po}\\
\frac{d}{dt}E_K(t) &= \int u\cdot u_t dx = -\int u_2\rho dx -\int |\nabla u|^2 dx.\label{Bou_Ki}
\end{align}
Summing them up, we obtain the conservation of the total energy:
\begin{align}\label{total_energy_inequality}
\frac{d}{dt}(E_P(t) + E_K(t)) = -\rVert \nabla u\rVert_{L^2}^2,\text{ thus  } E_T(t) + \int_{0}^t \rVert \nabla u\rVert_{L^2}^2 dt = E_T(0).
\end{align}

To derive an  estimate for $S(t)$, we need an  estimate for  $F$:
\begin{lemma}\label{Fst1}
There exists a universal constant $C>0$ such that
\[
\rVert F\rVert_{H^{-1}}\le C(\rVert \rho\rVert_{L^4} + \rVert u\rVert_{L^2})\rVert \nabla u\rVert_{L^2}.
\]
\end{lemma}
\begin{proof}
For any smooth test vector field $\varphi$, we have
\[
\int F\cdot \varphi dx = \int (u\cdot\nabla u )\cdot \varphi dx + \int v_t\cdot \varphi dx = -\int u\otimes u : \nabla \varphi dx + \int v_t\cdot \varphi dx,
\]
where the second equality is due to the integration by parts with the incompressibility and the boundary condition of $u$. Therefore, using the usual density argument of $C^{\infty}(\Omega)$ in $H^1(\Omega)$, we get
\[
\rVert F \rVert_{H^{-1}}\le \rVert u\otimes u\rVert_{L^2} + \rVert v_t\rVert_{H^{-1}}\le_C \rVert u\rVert_{L^4}^2 + \rVert v_t\rVert_{H^{-1}}.
\]
For the first term in the right-hand side, Ladyzhenskaya's inequality gives us that  
\begin{align}\label{estimate1p2}
\rVert u\rVert_{L^4}^2\le C \rVert u\rVert_{L^2} \rVert \nabla u\rVert_{L^2}\le C\rVert \nabla u\rVert_{L^2}^2,
\end{align}
where the last inequality follows from the Poincar\'e inequality.
We apply Lemma~\ref{Stokes_wellposed} to \eqref{sBvec23s22}, yielding that
\[
\rVert v_t\rVert_{H^{-1}}\le C \rVert v_t\rVert_{H^1}\le C \rVert \nabla\cdot (u\rho)\rVert_{H^{-1}}\le \rVert u\rho\rVert_{L^2}\le \rVert \rho\rVert_{L^4}\rVert u\rVert_{L^4}\le C\rVert \rho\rVert_{L^4}\rVert \nabla u \rVert_{L^2},
\]
where the last inequality is due to the estimate for $\rVert u\rVert_{L^4}$ obtained above.
Combining this with the first inequality in \eqref{estimate1p2}, we obtain the desired estimate for $F$.
\end{proof}

Now we are ready to prove the main result of this section, that is, a suitable linear combination of $S(t)$ and the total energy $E_T(t)$ is a Lyapunov functional.
\begin{proposition}\label{lyaypiunove}
There exist constants $C_1,C_2>0$ depending on $\rVert\rho_0\rVert_{L^4}$ and  $E_T(0)$ such that  
\[
\frac{d}{dt} \left(C_1E_T(t) + S(t)\right)\le -C_2\left(\rVert \nabla u\rVert_{L^2}^2 +\rVert \nabla w\rVert_{L^2}^2 + \rVert \nabla v\rVert_{L^2}^2 \right).
\]
\end{proposition}
\begin{proof}
Let us estimate the evolution of $S(t)$, using the equation for $w$ in \eqref{evol_w}:
\[
\frac{d}{dt}S(t) = \int w w_t dx = \int F \cdot w dx + \int w\cdot \Delta wdx\le \rVert \nabla w\rVert_{L^2}\rVert F\rVert_{H^{-1}} -\rVert \nabla w\rVert_{L^2}^2\le C \rVert F \rVert_{H^{-1}}^2 - \frac{1}{2}\rVert \nabla w\rVert_{L^2}^2,
\]
where the first inequality is due to the Poincar\'e inequality, $\rVert w\rVert_{H^1}\le_C\rVert\nabla w\rVert_{L^2}^2$, and  we used the Cauchy-Schwarz inequality in the last inequality. Then Lemma~\ref{Fst1} gives us
\begin{align*}
\frac{d}{dt}S(t)&\le  C(\rVert \rho(t)\rVert_{L^4} + \rVert u(t)\rVert_{L^2})^2 \rVert \nabla u\rVert_{L^2}^2 - C\rVert \nabla w\rVert_{L^2}^2\\
&\le C(1+\rVert \rho(t)\rVert_{L^4} + \rVert u(t)\rVert_{L^2})^2 \rVert \nabla u\rVert_{L^2}^2 - C\left(\rVert \nabla w\rVert_{L^2}^2+\rVert \nabla v\rVert_{L^2}^2 \right),
\end{align*}
where the last inequality is due to the following elementary application of the triangular inequality, 
\[
\rVert \nabla w\rVert_{L^2}^2\ge C(\rVert \nabla w\rVert_{L^2}^2 + \rVert \nabla(w-u)\rVert_{L^2}^2) - C\rVert \nabla u\rVert_{L^2}^2 \ge C(\rVert \nabla w\rVert_{L^2}^2 + \rVert \nabla v\rVert_{L^2}^2) - C\rVert \nabla u\rVert_{L^2}^2.
\]
Since $\rho$ is transported by  incompressible flows, we have a conservation of $\rVert\rho(t)\rVert_{L^p}$ for any $p\in [1,\infty]$. Especially we have $\rVert \rho(t)\rVert_{L^4}=\rVert \rho_0\rVert_{L^4}$.  Combining this with the energy conservation \eqref{total_energy_inequality}, we arrive at
\[
\frac{d}{dt}S(t)\le C(1+\rVert \rho_0\rVert_{L^4}+E_T(0))^2 \rVert \nabla u\rVert_{L^2}^2 - C\left(\rVert \nabla w\rVert_{L^2}^2+\rVert \nabla v\rVert_{L^2}^2\right).
\]
Therefore, we can choose $C_1$ large enough depending on $\rVert \rho_0\rVert_{L^4}$ and $E_T(0)$ so that the above inequality and \eqref{total_energy_inequality} give us
\begin{align*}
\frac{d}{dt}\left( C_1 E_T(t) + S(t)\right) &\le -\left(C_1 - C(1+\rVert \rho_0\rVert_{L^2}+E_T(0))^2\right)\rVert \nabla u\rVert_{L^2}^2 - C\left(\rVert \nabla w\rVert_{L^2}^2+\rVert \nabla v\rVert_{L^2}^2\right) \\
& \le -C_2\left(\rVert \nabla u\rVert_{L^2}^2  + \rVert \nabla w\rVert_{L^2}^2+\rVert \nabla v\rVert_{L^2}^2\right),
\end{align*}
for some constant $C_2>0$, which also depend only on  $\rVert \rho_0\rVert_{L^4}$ and $E_T(0)$. This finishes the proof.
\end{proof}
\section{Long-time  instability for an arbitrary  steady state}\label{ins}
In this section, we investigate an application of the Lyapunov functional derived in the earlier section yielding that a long-time instability of an arbitrary steady state.
 The main strategy is based on the proof of instability for the incompressible porous media equation in \cite[Section 4,5]{kiselev2023small} and its application to the Stokes transport system in \cite[Section 3]{leblond2023well}. 
  
   We first recall the definition of \textit{bubble type} density profile, which possesses a simple closed regular level curve.
   \begin{definition}\cite[Definition 3.7.3]{leblond2023well}
   A function $f:\Omega\mapsto \mathbb{R}$ is said to be of {bubble type}, if there exists a closed curve $\Gamma_0\subset \Omega$ enclosing a simply connected set such that $f$ is constant on $\Gamma_0$ and $\inf_{x\in\Gamma_0}|\nabla f| >0$.
   \end{definition}
   Noting that $H^{2-\gamma}$ for any $\gamma>0$ fails to embed into $C^1$, it is possible to design an arbitrarily small perturbation of any stratified density $\rho_s$ in $H^{2-\gamma}$ destroying its stratified level set structure. More precisely, we recall the following:
   \begin{lemma}\cite[Lemma 3.7.7]{leblond2023well}\label{bublel}
   For any stratified density $\rho_s\in C^\infty(\Omega)$ and $\epsilon,\gamma>0$, there exists $\rho_0\in C^\infty(\Omega)$ such that $\rVert \rho_s-\rho_0\rVert_{H^{2-\gamma}}\le \epsilon$ and $\rho_0$ is of bubble type.
   \end{lemma}   
   
   The following lemma addresses a uniform lower bound of an anisotropic norm of solutions to a general transport equation when the initial data is of bubble type. While the detailed proof provided in \cite{kiselev2023small} is presented in the context of the IPM equation, the proof remains essentially the same. Therefore, we will omit the proof here.
 \begin{lemma}\cite[Proposition 4.1, Corollary 4.2]{kiselev2023small}\label{kislem} Let $f$ be a  solution to the transport equation
 \[
 f_t + U\cdot\nabla f=0,\quad f(0,x)=f_0(x),
 \]
 for some smooth divergence free vector field $U$. If $f_0$ is a smooth function of bubble type, then  there exists a positive constant $C=C(f_0)$ such that
 \[
 \rVert \partial_1f(t)\rVert_{L^2}\ge C,\text{ for all $t>0$}.
 \]
 \end{lemma}
 Now,  the proof of Theorem~\ref{instability} follows straightforwardly.
 \begin{proof}[Proof of Theorem~\ref{instability}]
 Let $\rho_s$ be a smooth stratified density and let $\gamma,\epsilon>0$ be fixed. Let us fix a smooth initial velocity $u_0$ such that $\nabla\cdot u_0=0$ and $u_0=0$ on $\partial\Omega$. Thanks to Lemma~\ref{bublel}, we can find a bubble type smooth initial density $\rho_0$ such that $\rVert \rho_s-\rho_0\rVert_{H^{2-\gamma}}<\epsilon$. Since the initial data $(\rho_0,u_0)$, the Boussinesq system~\ref{Bouss} admits a unique smooth solution $(\rho(t),u(t))$, $\rho(t)$ is transported by a smooth divergence free vector field. Then Lemma~\ref{kislem} tells us that there exists a constant $C=C(\rho_0)>0$ such that 
 \begin{align}\label{unifformsd2}
 \rVert \partial_1\rho(t)\rVert_{L^2}\ge C,\text{ for all $t>0$.}
 \end{align}
 Moreover, Proposition~\ref{lyaypiunove} yields that for some constants $C_1,C_2$ depending on the initial data, we have 
 \begin{align}\label{energy_b}
 C_1E_T(t)+S(t) + C_2\int_{0}^t \rVert \nabla u(s)\rVert_{L^2}^2 +\rVert \nabla w(s)\rVert_{L^2}^2 +\rVert \nabla v(s)\rVert_{L^2}^2ds \le \underbrace{C_1E_T(0) + S(0)}_{=: C(\rho_0,u_0)}, \text{ for all $t>0$.}
 \end{align}
 Taking $t\to\infty$, we obtain
 \[
 \int_{0}^{\infty}\rVert \nabla v(s)\rVert_{L^2}^2  ds \le C(\rho_0,u_0).
 \]
 This implies that
 \begin{align}\label{limsup1}
 \liminf_{t\to\infty}t\rVert \nabla v(t)\rVert_{L^2}^2 = 0.
 \end{align} 
 Applying Lemma~\ref{rigorousnavier} to the steady Stokes equation \eqref{sBvec}, we observe that the velocity $v$ admits a stream function $\Psi$ such that $v=\nabla^\perp \Psi$ and
  \begin{align}\label{streamv2}
  \begin{cases}
  \Delta^2 \Psi = \partial_1 \rho & \text{ in $\Omega$}\\
  \Psi=\nabla \Psi = 0  & \text{ on $\partial\Omega$.}
  \end{cases}
  \end{align}
 Therefore,  the Gagliardo-Nirenberg interpolation inequality give us
 \[
 \rVert \partial_1\rho\rVert_{L^2}\le C \rVert \Delta^2\Phi\rVert_{L^2}\le C\rVert \Delta \Phi\rVert_{L^2}^{k/(k+2)}\rVert \Delta^2\Phi\rVert_{H^{k}(\Omega)}^{2/(k+2)}\le C\rVert \nabla v\rVert_{L^2}^{k/(k+2)}\rVert \partial_1\rho\rVert_{H^k}^{2/(k+2)},\text{ for all $k> 0$.}
 \]
 Combining this with \eqref{unifformsd2}, we obtain $
 \rVert \partial_1\rho(t) \rVert_{H^k}\ge C \rVert \nabla v(t)\rVert_{L^2}^{-k/2}$. Together with \eqref{limsup1}, we conclude that 
 \[
 \limsup_{t\to \infty} t^{-k/4} \rVert \partial_1\rho(t) \rVert_{H^k} \ge \limsup_{t\to\infty}\left(t\rVert \nabla v\rVert_{L^2}^2 \right)^{-k/4}=\infty, \text{ for all $k>0$,}
 \]
 which proves the theorem. 
 \end{proof}


\section{Asymptotic stability for vertically decreasing stratified density}\label{assssss}
In this section, we aim to prove asymptotic stability of hydrostatic steady states $({\rho}_s,0)$, where ${\rho}_s(x)={\rho}_s(x_2)$ is independent of the horizontal variable $x_1$. As stated in Theorem~\ref{main_Boussinesq}, we will  assume that $\rho_s$ satisfies \eqref{steadtysd2sd} throughout the section.  We also denote $C$ a positive constant that depends on only $\rho_s$, especially on $\rVert \partial_2\rho_s\rVert_{H^4}<\infty$ but the constant $C$  might vary from line to line. 

 Let us denote 
 \begin{align}\label{notationeavesd}
 \theta(t):=\rho(t) - {\rho}_s.
 \end{align}  In terms of $\theta$, we can rewrite \eqref{Bouss} as 
\begin{align}\label{Boussiesq_pertur}
\begin{cases}
\theta_t + u\cdot \nabla \theta = -\partial_2{\rho}_s u_2,\\
u_t + u\cdot\nabla u = -\nabla p - \colvec{0 \\ \theta}+ \Delta u,\\
\nabla\cdot u =0,
\end{cases}
\end{align}
where the pressure is suitably replaced using that $(0, {\rho}_s)^{T}$ is a gradient field.
Similarly, we can rewrite the steady Stokes equation for the  auxiliary vector field $v$  in \eqref{sBvec}  as
\begin{align}\label{sBvec1}
-\Delta v = -\nabla q -\colvec{0 \\ \theta}\text{ and } \nabla\cdot v=0 \text{ in $\Omega$},\quad v = 0 \text{ on $\partial\Omega$.}
\end{align} 
If the velocity $u$ is replaced by $v$ in the first equation in \eqref{Boussiesq_pertur} (i.e., the density $\rho$ in \eqref{Bouss} is transported by the vector field $v$ instead of $u$), the resulting active scalar equation is the so-called Stokes transport system, for which asymptotic stability was  proved near a specific equilibrium $\rho_s(x_2)=1-x_2$ in \cite[Theorem 1.1]{dalibard2023long}. The difference between $u$ and $v$ is obviously encoded in the quantity $w$ defined in \eqref{wdef_1db}. Hence before getting into detailed analysis for asymptotic stability, let us  derive some  estimates  for $w$ in  the next subsection.

\subsection{Estimates for derivatives of $w$}  Due to the presence of the boundary in the domain $\Omega$ and the viscous nature  of the fluid, we do not seek to estimate high Sobolev norms directly by taking spatial derivatives in the evolution equation. Instead, following the classical approach for the regularity theory of the Navier-Stokes equations in bounded domains (e.g. \cite[Chapter 5]{galdi2000introduction}), we study the time derivatives of $w$.

 Taking derivatives in time in \eqref{evol_w}, we observe that
\begin{align}\label{wtt_1}
w_{tt}  = - \nabla p_t - F_t +\Delta w_t, \quad \nabla \cdot w_t =0, \  \text{ in $\Omega$} \text{ and }w_t = 0  \text{ on $\partial \Omega$},
\end{align}
and
\begin{align}\label{wtt_2}
w_{ttt}  = - \nabla p_{tt} - F_{tt} +\Delta w_{tt}, \quad \nabla \cdot w_{tt} =0,\   \text{ in $\Omega$}, \text{ and }w_{tt} = 0  \text{ on $\partial \Omega$}.
\end{align}
 
  We will first prove several lemmas showing that the time derivatives of the vector vields can actually control spatial derivatives of vector fields, if we have enough information of $F$. The necessary estimates for $F$  will be studied afterwards.

  Also note that because our main interest lies in the solutions sufficiently close to the equilibrium $(\rho_s,0)$, we will only focus on $\theta$ and $u$ that are sufficiently small, which reduces the complication in computations.  More precisely, we will assume  that 
  \begin{align}\label{size_assu}
\rVert w_{tt}(t)\rVert_{L^2}^2 +\rVert w_t(t)\rVert_{L^2}^2+ \rVert u(t)\rVert_{H^4}^2+\rVert \theta(t)\rVert_{H^4}^2 + \int_0^t \rVert\partial_1\Delta \theta\rVert_{L^2}^2 +\rVert u(t)\rVert_{H^5}^2+\rVert w(t)\rVert_{H^5}^2 dt\le \delta\le \delta_0(\rho_s)\ll 1,
\end{align}
 for $t\le [0,T]$ for some $T>0$ and for some $\delta_0>0$ that depends only on $\rho_s$.
We emphasize that the implicit constant $C$ that will appear throughout the proofs will not depend on $\delta_0$.

We will frequently use the following Sobolev embeddings:
\begin{align}\label{embd1bo}
W^{k,\infty}(\Omega)\hookrightarrow H^{k+2}(\Omega), \text{ for $n\in\mathbb{N}\cup\left\{ 0\right\}$},
\end{align}
and the usual tame estimate:
\[
\rVert fg \rVert_{H^k}\le \rVert f\rVert_{L^\infty}\rVert g\rVert_{H^k} + \rVert f\rVert_{H^k}\rVert g\rVert_{L^\infty}, \text{ for $k\ge 0$,}
\]
without specifically mentioning in the proof.

\begin{lemma}\label{leray_lem1}
Let $w$ be a smooth solution to \eqref{evol_w}. Then it holds that
\begin{align}
\rVert w\rVert_{H^4} \le_C  \rVert \mathbb{P}\Delta w_t\rVert_{L^2} + \rVert F_t\rVert_{L^2} +  \rVert F\rVert_{H^2}, \label{leaving1} \\
\rVert w\rVert_{H^4} \le_C \rVert w_{tt}\rVert_{L^2} +\rVert F_t\rVert_{L^2}+ \rVert F\rVert_{H^2}, \label{leaving2}\\
\rVert w\rVert_{H^5}  \le_C \rVert w_{tt}\rVert_{H^1} +\rVert F\rVert_{H^3} + \rVert F_t\rVert_{H^1}, \label{leaving3}
\end{align}
where $\mathbb{P}$ denotes the Leray projection defined in Lemma~\ref{leray_pro}.
\end{lemma}
\begin{proof}
 Firstly, taking the Leray projection $\mathbb{P}$ in the evolution equation in \eqref{wtt_1}, we have
\[
w_{tt} = -\mathbb{P}F_t +\mathbb{P}\Delta w_t.
\] 
Hence using the estimate in Lemma~\ref{leray_pro}, we get
\[
\rVert w_{tt}\rVert_{L^2}\le_C \rVert \mathbb{P}F_t\rVert_{L^2} + \rVert \mathbb{P}\Delta w_t\rVert_{L^2}\le_C \rVert F_t\rVert_{L^2} + \rVert \mathbb{P}\Delta w_t\rVert_{L^2}.
\]
Meanwhile, applying Lemma~\ref{Stokes_wellposed} to \eqref{wtt_1}, we also have $
\rVert w_t \rVert_{H^2} \le \rVert w_{tt}\rVert_{L^2} + \rVert F_t\rVert_{L^2},$ yielding that
\begin{align}\label{rkawk1ppoross}
\rVert w_t \rVert_{H^2} \le \rVert F_t\rVert_{L^2} + \rVert \mathbb{P}\Delta w_t\rVert_{L^2}.
\end{align}
Then applying Lemma~\ref{Stokes_wellposed} again to \eqref{evol_w}, we find  $\rVert w\rVert_{H^4}\le \rVert w_t\rVert_{H^2}+\rVert F\rVert_{H^2}$. Combining this with \eqref{rkawk1ppoross}, we obtain \eqref{leaving1}. 

 Next, we apply Lemma~\ref{Stokes_wellposed} to \eqref{wtt_1} and \eqref{evol_w}, which gives us that for $m\ge0$,
\begin{align}\label{ssqhrhtlvdsmwanja}
\rVert w_t\rVert_{H^{k+2}}\le \rVert w_{tt}\rVert_{H^k} +\rVert F_t\rVert_{H^k},\quad \rVert w\rVert_{H^{k+2}} \le \rVert w_t\rVert_{H^{k}}+\rVert F\rVert_{H^{k}}.
\end{align}
Plugging  $k=0,1,2$ into  these  estimates gives us \eqref{leaving2} and \eqref{leaving3}.
\end{proof}

\begin{lemma}\label{timeinu}
Suppose that $(\theta,u)$  is a smooth solution to \eqref{Boussiesq_pertur}. Then there exists $\delta_0=\delta(\rho_s)>0$ such that if \eqref{size_assu} holds for $\delta_0$, then the following estimates hold:
\begin{align}
\rVert v_t\rVert_{H^2}&\le_C \rVert u\rVert_{L^2},\label{timeinv2}\\
\rVert \theta_t\rVert_{H^1}&\le_C \rVert u\rVert_{H^1}, \label{ijtjsdskd1}\\
\rVert v_t\rVert_{H^3}&\le_C \rVert u\rVert_{H^1},\label{timeinv31}\\
\rVert u_t\rVert_{L^2}&\le_C \rVert w\rVert_{H^2}+\rVert F\rVert_{L^2}+\rVert u\rVert_{L^2},\label{timeinu1}\\
\rVert u_t\rVert_{H^2}&\le_C \rVert w_{tt}\rVert_{L^2} + \rVert F_t\rVert_{L^2} +\rVert u\rVert_{L^2},\label{timeinu3}\\
\rVert \theta_{tt}\rVert_{H^{-1}}&\le_C \rVert w\rVert_{H^2}+\rVert F\rVert_{L^2}+\rVert u\rVert_{L^2}, \label{timeint1}\\
\rVert \theta_{tt}\rVert_{L^2}&\le_C \rVert w\rVert_{H^2}+\rVert F\rVert_{L^2}+\rVert u\rVert_{H^1}, \label{ijtjsdskd2}\\
\rVert v_{tt}\rVert_{H^1}&\le_C \rVert w\rVert_{H^2}+\rVert F\rVert_{L^2}+\rVert u\rVert_{L^2}, \label{timeinv3}\\
\rVert u_{tt}\rVert_{L^2}&\le_C \rVert w_{tt}\rVert_{L^2}+ \rVert w\rVert_{H^2}+\rVert F\rVert_{L^2}+\rVert u\rVert_{L^2},\label{timeinu4}\\
\rVert v_{ttt}\rVert_{H^{-1}}&\le_C\rVert w_{tt}\rVert_{L^2}+\rVert w\rVert_{H^2}+\rVert F_t\rVert_{L^2}+ \rVert F\rVert_{L^2}+\rVert u\rVert_{H^1} \label{timeint2}.
\end{align}
\end{lemma}
\begin{proof}

\textbf{Proof of \eqref{timeinv2}.}  
We differentiate the equation \eqref{sBvec1} in  time to see that
\begin{align}\label{sBvecs2d2d3ss22}
-\Delta v_t = -\nabla q_t -\colvec{0\\\theta_t} =-\nabla q_t -\colvec{0 \\ - u \cdot\nabla \theta  -  \partial_2\rho_s u_2},\quad 
\nabla\cdot v_t=0, \text{ in $\Omega$ and },
v_t = 0 \text{ on $\partial\Omega$.}
\end{align}
Applying Lemma~\ref{Stokes_wellposed} to \eqref{sBvecs2d2d3ss22} yields that
\begin{align}\label{rksdjwo2sdxcw2}
\rVert v_t\rVert_{H^2}\le_C \rVert u\cdot\nabla\theta\rVert_{L^2}+\rVert \partial_2\rho_s u_2\rVert_{L^2}\le_C (\rVert \nabla\theta\rVert_{L^\infty}+\rVert \partial_2\rho_s\rVert_{L^\infty})\rVert u\rVert_{L^2}\le_C \rVert u\rVert_{L^2},
\end{align}
where we used \eqref{size_assu} and \eqref{steadtysd2sd} for the last inequality.

\textbf{Proof of \eqref{ijtjsdskd1}.} From $\theta_t$ in \eqref{Boussiesq_pertur}, we compute
\[
\rVert \theta_t\rVert_{H^1}\le_C \rVert u\nabla\theta\rVert_{H^1}+\rVert \partial_2\rho_su_2\rVert_{H^1}\le_C\rVert u\rVert_{H^1},
\]
where the last inequality is due to \eqref{size_assu} and \eqref{steadtysd2sd}.

\textbf{Proof of \eqref{timeinv31}.}  
Again applying Lemma~\ref{Stokes_wellposed} to \eqref{sBvecs2d2d3ss22}, we get $\rVert v_t\rVert_{H^3}\le_C\rVert \theta_t\rVert_{H^1}\le_C \rVert u\rVert_{H^1}$, where the last inequality is due to \eqref{ijtjsdskd1}.

\textbf{Proof of \eqref{timeinu1}.}  Using $w=u-v$, we get
\[
\rVert u_t\rVert_{L^2}\le \rVert w_t\rVert_{L^2} + \rVert v_t\rVert_{L^2}\le_C \rVert w\rVert_{H^2} + \rVert F\rVert_{L^2} + \rVert v_t\rVert_{L^2},
\]where the last inequality follows from applying Lemma~\ref{leray_pro} to \eqref{evol_w}. Combining this with \eqref{timeinv2}, we obtain \eqref{timeinu1}.

\textbf{Proof of \eqref{timeinu3}.} Again using $w=u-v$, we  estimate
\[
\rVert u_t\rVert_{H^2}\le \rVert w_t\rVert_{H^2} + \rVert v_t\rVert_{H^2}\le \rVert w_t\rVert_{H^2}+ \rVert u\rVert_{L^2},
\]
where the last inequality is due to \eqref{timeinv2}. Applying Lemma~\ref{Stokes_wellposed} to \eqref{wtt_1}, we get 
\[
\rVert w_t\rVert_{H^2}\le_C \rVert w_{tt}\rVert_{L^2} +\rVert F_t\rVert_{L^2},
\]
which gives the desired result.

\textbf{Proof of \eqref{timeint1}.}  
Differentiating the first equation in \eqref{Boussiesq_pertur}, we see that
\begin{align}\label{thet11cpp}
\theta_{tt} = -u_t\cdot\nabla \theta -u\cdot\nabla \theta_t -\partial_2\rho_s \partial_t u_2= -\nabla\cdot(u_t\theta + u\theta_t) - \partial_2\rho_s \partial_t u_2.
\end{align}
Hence,
\begin{align*}
\rVert \theta_{tt}\rVert_{H^{-1}}&\le_C \rVert u_t\theta + u\theta_t\rVert_{L^2} + \rVert \partial_2\rho_s\partial_2u_2\rVert_{L^2}\\
&\le_C \rVert u_t\rVert_{L^2} + \rVert \theta_t\rVert_{L^2} + \rVert u_2\rVert_{L^2}.
\end{align*}
where the last inequality is due to \eqref{size_assu} and \eqref{steadtysd2sd}. Moreover, from the first equation in \eqref{Boussiesq_pertur}, we observe that
\[
\rVert \theta_t\rVert_{L^2}\le_C \rVert u\cdot \nabla \theta\rVert_{L^2} + \rVert \partial_2\rho_s u_2\rVert_{L^2}\le_C \rVert u\rVert_{L^2},
\]
where the last inequality is due to \eqref{size_assu} and \eqref{steadtysd2sd}. Hence, we arrive at $\rVert \theta_{tt}\rVert_{H^{-1}}\le_C \rVert u_t\rVert_{L^2}+ \rVert u\rVert_{L^2}$. Combining this with \eqref{timeinu1}, we obtain the desired result.

\textbf{Proof of \eqref{ijtjsdskd2}.} Using \eqref{thet11cpp}, we compute
\[
\rVert \theta_{tt}\rVert_{L^2}\le_C \rVert u_t\cdot\nabla \theta\rVert_{L^2}+\rVert u\cdot\nabla \theta_t\rVert_{L^2}+\rVert\partial_2\rho_s\partial_tu_2\rVert_{L^2}\le_C \rVert u_t\rVert_{L^2} + \rVert \theta_t\rVert_{H^1}\le_C \rVert w\rVert_{H^2}+\rVert F\rVert_{L^2}+\rVert u\rVert_{H^1},
\]
where the second inequality is due to  \eqref{size_assu} and \eqref{steadtysd2sd}, and the last inequality follows from \eqref{ijtjsdskd1} and \eqref{timeinu1}.

\textbf{Proof of \eqref{timeinv3}.}  
Observe from \eqref{sBvec1} that $\partial_{tt}v$ solves
\begin{align}\label{sBvec32}
-\Delta v_{tt} = -\nabla q_{tt} -\colvec{0 \\ \theta_{tt}},\quad \nabla\cdot v_{tt}=0 \text{ in $\Omega$,}\quad v_{tt} = 0 \text{ on $\partial\Omega$.}
\end{align}
Applying Lemma~\ref{Stokes_wellposed}, we see that
\begin{align}\label{vttpreg}
\rVert v_{tt}\rVert_{H^{k+2}}\le_C \rVert \theta_{tt}\rVert_{H^k}, \text{ for $k\ge -1$.}
\end{align}
Plugging $k=-1$ and using \eqref{timeint1}, we get \eqref{timeinv3}.

\textbf{Proof of \eqref{timeinu4}.}   Using $w=u-v$, we have
\[
\rVert u_{tt}\rVert_{L^2}\le_C \rVert w_{tt}\rVert_{L^2} + \rVert v_{tt}\rVert_{L^2}\le_C \rVert w_{tt}\rVert_{L^2} +  \rVert w\rVert_{H^2}+\rVert F\rVert_{L^2}+\rVert u\rVert_{L^2},
\]
where the last inequality is due to \eqref{timeinv3}.

\textbf{Proof of \eqref{timeint2}.} Differentiating \eqref{thet11cpp} in time, we get
\[
\theta_{ttt}= \nabla\cdot (u_{tt}\theta + 2u_{t}\theta_t + u\theta_{tt}) -\partial_2\rho_s\partial_{tt}u_2.
\]
Thus, from \eqref{steadtysd2sd} and \eqref{size_assu}, it follows that
\begin{align*}
\rVert \theta_{ttt}\rVert_{H^{-1}}&\le_C \rVert u_{tt}\theta + 2u_{t}\theta_t+u\theta_{tt}\rVert_{L^2} + \rVert \partial_{tt}u_2\rVert_{L^2}\\
&\le_C \rVert u_{tt}\rVert_{L^2}+\rVert u_t\rVert_{L^\infty}\rVert \theta_t\rVert_{L^2}+\rVert u\rVert_{L^\infty}\rVert \theta_{tt}\rVert_{L^2}\\
&\le_C \rVert u_{tt}\rVert_{L^2} + \rVert u_{t}\rVert_{H^2}\rVert u\rVert_{H^1}+\rVert u\rVert_{L^\infty}\left(\rVert w\rVert_{H^2}+\rVert F\rVert_{L^2}+\rVert u\rVert_{H^1}\right),
\end{align*}
where the last inequality is due to \eqref{ijtjsdskd1} and \eqref{ijtjsdskd2}. Using \eqref{size_assu}, \eqref{timeinu3} and \eqref{timeinu4}, we notice that the above estimate implies
\begin{align}\label{wanja1123}
\rVert \theta_{ttt}\rVert_{H^{-1}}\le_C \rVert w_{tt}\rVert_{L^2}+ \rVert F_t\rVert_{L^2}+\rVert w\rVert_{H^2}+\rVert F\rVert_{L^2}+\rVert u\rVert_{H^1}.
\end{align}
Observe that differentiating \eqref{sBvec32} in time, we have
 \begin{align}\label{sBvec322sdsd}
-\Delta v_{ttt} = -\nabla q_{ttt} -\colvec{0 \\ \theta_{ttt}},\quad \nabla\cdot v_{ttt}=0 \text{ in $\Omega$,}\quad v_{ttt} = 0 \text{ on $\partial\Omega$.}
\end{align}
Hence, applying Lemma~\ref{Stokes_wellposed}, we get
\[
\rVert v_{ttt}\rVert_{H^{-1}}\le_C\rVert v_{ttt}\rVert_{H^1}\le_C \rVert \theta_{ttt}\rVert_{H^{-1}}.
\]
Hence the estimate \eqref{wanja1123} implies the desired result.
\end{proof}

Now, we derive necessary estimates for $F$. 

\begin{lemma}\label{endless_estimates1}
Suppose that $(\theta,u)$  is a smooth solution to \eqref{Boussiesq_pertur}. Then there exists $\delta_0=\delta(\rho_s)>0$ such that if \eqref{size_assu} holds for $\delta_0$, then the following estimates hold:
\begin{align}
\rVert F\rVert_{L^2}&\le_C \rVert u\rVert_{H^1},\label{BouFestimate11}\\
\rVert F\rVert_{H^2}&\le_C \rVert u\rVert_{H^1},\label{BouFestimate1}\\
\rVert F\rVert_{H^3}&\le_C \rVert u\rVert_{H^3}, \label{BouFestimate2}\\
\rVert F_t\rVert_{L^2}&\le_C  \rVert w\rVert_{H^2}+\rVert u\rVert_{H^1} + \rVert u\rVert_{H^3}\rVert w_{tt}\rVert_{L^2}, \label{BouFestimate3}\\
\rVert F_t\rVert_{H^1}&\le_C \rVert u\rVert_{H^4} \rVert w_{tt}\rVert_{L^2} + \rVert w\rVert_{H^2}+\rVert u\rVert_{H^1},\label{BouFestimate4}\\
\rVert F_{tt}\rVert_{H^{-1}}&\le_C \rVert w_{tt}\rVert_{L^2}+ \rVert w\rVert_{H^2}+\rVert u\rVert_{H^1}.\label{BouFestimate5}
\end{align}
\end{lemma}
\begin{proof}
\textbf{Proof of \eqref{BouFestimate11}.}  Recalling $F$ from \eqref{wdef_1db},
\begin{align}\label{wdef_1db2}
F=u\cdot \nabla u + v_t,
\end{align}
we  estimate
\[
\rVert F\rVert_{L^2}\le \rVert v_t\rVert_{L^2} + \rVert u\cdot\nabla u\rVert_{L^2}\le_C \rVert v_t\rVert_{L^2}+ \rVert u\rVert_{L^\infty}\rVert u\rVert_{H^1}\le_C\rVert v_t\rVert_{L^2}+\rVert u\rVert_{H^1}.
\] where the last inequality is due to \eqref{size_assu}. Using \eqref{timeinv2} we obtain the desired result.

\textbf{Proof of \eqref{BouFestimate1}.} 
In view of \eqref{wdef_1db2}, we  estimate 
\[
\rVert F\rVert_{H^2}\le_C \rVert v_t\rVert_{H^2}+\rVert u\cdot\nabla u\rVert_{H^2}\le_C  \rVert v_t\rVert_{H^2} + \rVert \nabla (u\otimes u)\rVert_{H^2}\le_C \rVert v_t\rVert_{H^2} + \rVert u\otimes u\rVert_{H^3}.
\]
Using the tame estimate, we have
\[
 \rVert u\otimes u\rVert_{H^3}\le_C \rVert u\rVert_{L^\infty}\rVert u\rVert_{H^3}\le_C \rVert u\rVert_{H^2}\rVert u\rVert_{H^3}.
\]
The Gagliardo-Nirenberg interpolation inequality tells us
\[
\rVert u\rVert_{H^2}\le_C \rVert u\rVert_{H^1}^{2/3}\rVert_{H^4}^{1/3},\quad \rVert u\rVert_{H^3}\le_C \rVert u\rVert_{H^1}^{1/3}\rVert u \rVert_{H^4}^{2/3},
\]
therefore $\rVert u\otimes u\rVert_{H^3}\le_C \rVert u\rVert_{H^1}\rVert u\rVert_{H^4}\le_C \rVert u\rVert_{H^1}$,
where the last inequality is due to \eqref{size_assu}.  Using \eqref{timeinv2} we obtain
\[
\rVert F\rVert_{H^2}\le_C \rVert u\rVert_{H^1},
\] which is the desired result.

\textbf{Proof of \eqref{BouFestimate2}.} 
Again from the definition of $F$ in \eqref{wdef_1db2}, we need to estimate
\[
\rVert F\rVert_{H^3}\le \rVert  v_t\rVert_{H^3} + \rVert u\cdot\nabla u\rVert_{H^3}.
\]
Since $H^3(\Omega)$ is a Banach algebra, we have
\begin{align}\label{juylsd2}
\rVert u\cdot\nabla u\rVert_{H^3}\le \rVert u\rVert_{H^3}\rVert \nabla u\rVert_{H^3}\le \rVert u\rVert_{H^3}\rVert  u\rVert_{H^4}\le_C \rVert u\rVert_{H^3},
\end{align}
where the last inequality is due to \eqref{size_assu}. Hence, combining this with \eqref{timeinv31}, we obtain the desired result.

\textbf{Proof of \eqref{BouFestimate3}.} Differentiating \eqref{wdef_1db2} in time, we get
\begin{align}\label{Fdwppoq}
F_t = v_{tt} + u_t\cdot\nabla u + u\cdot\nabla u_t.
\end{align}
Hence, we  estimate
\[
\rVert F_t\rVert_{L^2}\le_C \rVert v_{tt}\rVert_{L^2}+\rVert u_t\cdot \nabla u\rVert_{L^2}+\rVert u\cdot \nabla u_t\rVert_{L^2}\le_C \rVert v_{tt}\rVert_{L^2}+\rVert u\rVert_{W^{1,\infty}}\rVert \nabla u_t\rVert_{H^1}\le_C \rVert v_{tt}\rVert_{L^2}+\rVert u\rVert_{H^3} \rVert u_{t}\rVert_{H^2}.
\]
Using \eqref{timeinv3} and \eqref{timeinu3}, we get
\begin{align*}
\rVert F_t\rVert_{L^2}&\le_C  \rVert w\rVert_{H^2}+\rVert F\rVert_{L^2}+\rVert u\rVert_{L^2} + \rVert u\rVert_{H^3}\left( \rVert w_{tt}\rVert_{L^2} +  \rVert F_t\rVert_{L^2}+\rVert u\rVert_{L^2}\right)\\
&\le_C \rVert w\rVert_{H^2}+\rVert F\rVert_{L^2}+\rVert u\rVert_{L^2} + \rVert u\rVert_{H^3}\rVert w_{tt}\rVert_{L^2} +  \sqrt{\delta_0} \rVert F_t\rVert_{L^2},
\end{align*}
where the last inequality is due to \eqref{size_assu}. Since $\delta_0$ is assumed small, $\delta_0\ll 1$ as noted in \eqref{size_assu}, we can send the last term $\sqrt{\delta_0} \rVert F_t\rVert_{L^2}$ to the left-hand side and obtain 
\[
\rVert F_t\rVert_{L^2}\le   \rVert w\rVert_{H^2}+\rVert F\rVert_{L^2}+\rVert u\rVert_{L^2} + \rVert u\rVert_{H^3}\rVert w_{tt}\rVert_{L^2}. 
\]
Then the desired result follows from \eqref{BouFestimate11}.

\textbf{Proof of \eqref{BouFestimate4}.} 
From \eqref{Fdwppoq}, we have
\begin{align}\label{112sosimple}
\rVert F_t\rVert_{H^1}\le \rVert \partial_{tt}v\rVert_{H^1} + \rVert u_t\cdot\nabla u\rVert_{H^1} + \rVert u\cdot\nabla u_t\rVert_{H^1}.
\end{align}
 As before, we estimate
\begin{align}\label{mo1222vethetatt1}
\rVert u_t\cdot\nabla u\rVert_{H^1} + \rVert u\cdot\nabla u_t\rVert_{H^1}\le \rVert u_t\rVert_{H^2}\rVert u\rVert_{W^{2,\infty}} \le \rVert u\rVert_{H^4}\rVert u_t\rVert_{H^2}\le \rVert u\rVert_{H^4}\rVert u_t\rVert_{H^2}.
\end{align}
 Thanks to \eqref{timeinu3}, this estimate implies
\[
\rVert u_t\cdot\nabla u\rVert_{H^1} + \rVert u\cdot\nabla u_t\rVert_{H^1}\le_C \rVert u\rVert_{H^4}\left( \rVert w_{tt}\rVert_{L^2} + \rVert F_t\rVert_{L^2}+\rVert u\rVert_{L^2}\right)\le_C \rVert u\rVert_{H^4} \rVert w_{tt}\rVert_{L^2} + \rVert w\rVert_{H^2}+\rVert u\rVert_{H^1},
\]
where the last inequality is due to \eqref{BouFestimate3} and \eqref{size_assu}. Combining this and \eqref{timeinv3} and plugging them into \eqref{112sosimple}, we obtain 
\[
\rVert F_t\rVert_{H^1}\le_C  \rVert u\rVert_{H^4} \rVert w_{tt}\rVert_{L^2} + \rVert w\rVert_{H^2}+\rVert u\rVert_{H^1} + \rVert F\rVert_{L^2}\le_C \rVert u\rVert_{H^4} \rVert w_{tt}\rVert_{L^2} + \rVert w\rVert_{H^2}+\rVert u\rVert_{H^1},
\]
where the last inequality is due to \eqref{BouFestimate11}. This is the desired estimate.

\textbf{Proof of \eqref{BouFestimate5}.}  From \eqref{Fdwppoq}, we get
\begin{align}\label{bathes}
F_{tt}= v_{ttt} + u_{tt}\cdot\nabla u + 2u_t\cdot\nabla u_{t} + u\cdot\nabla u_{tt}.
\end{align}
Hence, we estimate
\begin{align}\label{whasd1sdsx1}
\rVert F_{tt}\rVert_{H^{-1}}\le_C \rVert  v_{ttt}\rVert_{H^{-1}} + \rVert u_{tt}\otimes u + 2u_t\otimes u_t + u\otimes u_{tt}\rVert_{L^2}.
\end{align}
The second term can be estimated as
\begin{align}\label{whansdle1}
\rVert u_{tt}\otimes u + 2u_t\otimes u_t + u\otimes u_{tt}\rVert_{L^2}&\le_C \rVert u\rVert_{L^\infty}\rVert u_{tt}\rVert_{L^2} + \rVert u_t\rVert_{L^4}^2\le_C \rVert u\rVert_{L^\infty}\rVert u_{tt}\rVert_{L^2} + \rVert u_t\rVert_{H^1}^2\nonumber\\
&\le_C \rVert u_{tt}\rVert_{L^2} +\rVert u_t\rVert_{H^1}^2,
\end{align}
where the second inequality follows from the Sobolev embedding $H^1(\Omega)\hookrightarrow L^4(\Omega)$.
Note that \eqref{timeinu4} implies
\begin{align}\label{wnajawhen}
\rVert u_{tt}\rVert_{L^2}\le_C \rVert w_{tt}\rVert_{L^2}+ \rVert w\rVert_{H^2}+\rVert F\rVert_{L^2}+\rVert u\rVert_{L^2}\le_C \rVert w_{tt}\rVert_{L^2}+ \rVert w\rVert_{H^2}+\rVert u\rVert_{H^1},
\end{align}
where the last inequality is due to \eqref{BouFestimate11}. Furthermore, \eqref{timeinu3} implies
\begin{align}\label{sdk1posd21}
\rVert u_t\rVert_{H^1}&\le_C \rVert w_{tt}\rVert_{L^2}  +\rVert u\rVert_{L^2} +  \rVert F_t\rVert_{L^2}\nonumber\\
&\le_C \rVert w_{tt}\rVert_{L^2} + \rVert w\rVert_{H^2}+\rVert u\rVert_{H^1} + \rVert u\rVert_{H^3}\rVert w_{tt}\rVert_{L^2}\nonumber\\
&\le_C  \rVert w_{tt}\rVert_{L^2} + \rVert w\rVert_{H^2}+\rVert u\rVert_{H^1}.
\end{align}
where the second  inequality is due to \eqref{BouFestimate3} and the last inequality is due to \eqref{size_assu}. Using this and \eqref{wnajawhen}, we notice that \eqref{whansdle1} becomes
\begin{align}\label{tt1sdsd1}
\rVert u_{tt}\otimes u + 2u_t\otimes u_t + u\otimes u_{tt}\rVert_{L^2}\le_C \rVert w_{tt}\rVert_{L^2}+ \rVert w\rVert_{H^2}+\rVert u\rVert_{H^1}.
\end{align}
Also, it follows from \eqref{timeint2} that
\[
\rVert v_{ttt}\rVert_{H^{-1}}\le_C\rVert w_{tt}\rVert_{L^2}+\rVert w\rVert_{H^2}+\rVert F_t\rVert_{L^2}+ \rVert F\rVert_{L^2}+\rVert u\rVert_{H^1} \le_C \rVert w_{tt}\rVert_{L^2}+ \rVert w\rVert_{H^2}+\rVert u\rVert_{H^1},
\]
where the last inequality is due to \eqref{BouFestimate3} and \eqref{BouFestimate11}.
Thus, plugging this and \eqref{tt1sdsd1}  into \eqref{whasd1sdsx1}, we obtain the desired result.
\end{proof}

 Before we close this subsection, we prove that in order to ensure the smallness assumption for $\rVert u(t)\rVert_{H^4}$ in \eqref{size_assu} throughout the evolution of the solution, it is sufficient to control $\rVert w_{tt}(t)\rVert_{L^2}, \rVert w_t(t)\rVert_{L^2}$ and $\rVert \theta(t)\rVert_{H^4}$.
  \begin{lemma}\label{smallenough}
 Suppose $(\theta,u)$ is a  solution to \eqref{Boussiesq_pertur} with initial data $(\theta_0, u_0)$. 
Then there exists $\delta_0=\delta(\rho_s)>0$ such that if \eqref{size_assu} holds for $\delta_0$, then
\[
\rVert u(t)\rVert_{H^4}\le_C \rVert w_{tt}(t)\rVert_{L^2} +  \rVert w_t(t)\rVert_{L^2}+ \rVert \theta(t)\rVert_{H^4}  + E_T(0).
\] 

 \end{lemma}
\begin{proof}
From $w=u-v$, it is trivial that $\rVert u\rVert_{H^4}\le_C \rVert w\rVert_{H^4}+\rVert v\rVert_{H^4}$. Applying Lemma~\ref{Stokes_wellposed} to \eqref{sBvec1}, we have $\rVert v\rVert_{H^4}\le_C \rVert \theta\rVert_{H^2}\le_C\rVert \theta\rVert_{H^4}$. Using this and \eqref{leaving2}, we get
\begin{align}\label{alichandlern}
\rVert u\rVert_{H^4}&\le_C \rVert w_{tt}\rVert_{L^2} + \rVert F_t\rVert_{L^2} +\rVert F\rVert_{H^2}+\rVert \theta\rVert_{H^4}\nonumber\\
& \le_C \rVert w_{tt}\rVert_{L^2} + \rVert \theta\rVert_{H^4} + \rVert w\rVert_{H^2}+\rVert u\rVert_{H^1} + \rVert u\rVert_{H^3}\rVert w_{tt}\rVert_{L^2},
\end{align}
where the last inequality follows from \eqref{BouFestimate1} and \eqref{BouFestimate3}. Using the second inequality in \eqref{ssqhrhtlvdsmwanja} with $k=0$, we get
\begin{align}\label{rjksdpwowed2cumba}
\rVert w\rVert_{H^2}&\le_C \rVert w_t\rVert_{L^2} + \rVert F\rVert_{L^2}\le_C \rVert w_t\rVert_{L^2} + \rVert u\rVert_{H^1},
\end{align}
where the last inequality is due to again \eqref{BouFestimate11}.  Plugging this into \eqref{alichandlern}, we arrive at
\begin{align}\label{howareyou}
\rVert u\rVert_{H^4}\le_C \rVert w_{tt}\rVert_{L^2} + \rVert \theta\rVert_{H^4} + \rVert w_t\rVert_{L^2}+\rVert u\rVert_{H^1}+ \rVert u\rVert_{H^3}\rVert w_{tt}\rVert_{L^2} \le_C \rVert w_{tt}\rVert_{L^2} + \rVert \theta\rVert_{H^4} + \rVert w_t\rVert_{L^2}+\rVert u\rVert_{H^1},
\end{align}
where the last inequality is due to \eqref{size_assu}. Then using the following interpolation theorem (see \cite[Theorem 5.2]{adams2003sobolev}),
\[
\rVert u\rVert_{H^1}\le_C \left(  C_\eta\rVert u\rVert_{L^2} + \eta\rVert u\rVert_{H^4}\right),\text{ for all $\eta>0$,}
\]
we can choose $\eta$ sufficiently small so that \eqref{howareyou} reads 
\[
\rVert u\rVert_{H^4}\le_C  \rVert w_{tt}\rVert_{L^2} + \rVert \theta\rVert_{H^4} + \rVert w_t\rVert_{L^2} + \rVert u\rVert_{L^2}.
\]
Meanwhile, thanks to \eqref{total_energy_inequality}, we have $\rVert u(t)\rVert_{L^2}\le_C E_T(t)\le E_T(0)$. Plugging this into the above estimate, we finish the proof.
\end{proof}

\subsection{Evolution of high Sobolev norms}
In this subsection, we derive estimates for high Sobolev norms for $\theta$ and $w$. As mentioned in the beginning of the previous subsection, we will estimate evolution of the time derivatives of $w$ instead of the spatial derivatives, and this is due to the presence of viscosity. For $\theta$ as well, the presence of viscosity of fluids may induce non-zero vertical derivatives of $\theta$, even if the support of $\theta_0$ is away from the boundary as pointed out in \cite[Section 2.1]{dalibard2023long}. The following lemma gives  sufficient control of $\theta$ near the boundary.
\begin{lemma}\cite[Lemma 2.1]{dalibard2023long}\label{vanishing_boundary_boy}
Let $(\theta(t),u(t))$ be a  solution to \eqref{Boussiesq_pertur} such that $\theta_0\in H^2_0(\Omega)\cap H^4(\Omega)$.  Then for all $t>0$, it holds that
\[
\theta(t)=\partial_2\theta(t) =\partial_{22}\overline{\theta}(t)=0,\text{ on $\partial\Omega$,}
\]
where $\overline{\theta}(x_2):=\frac{1}{2\pi}\int_\mathbb{T}\theta(x_1,x_2)dx_1$.
\end{lemma}
\begin{proof}
The proof of the lemma is identical to \cite[Lemma 2.1]{dalibard2023long}. 
\end{proof}



The main result of this subsection is the next proposition. 
\begin{proposition}\label{bouenergy_es}
Suppose that a solution $(\theta(t),u(t))$ to \eqref{Boussiesq_pertur}.  Then there exists $\delta_0=\delta(\rho_s)>0$ such that if \eqref{size_assu} holds for $\delta_0$, then it holds that
\begin{align}\label{rakwk123bo}
\frac{1}{2}\frac{d}{dt}\left(\rVert \Delta^2\theta\rVert_{L^2}^2+\rVert w_t\rVert_{L^2}^2+\rVert w_{tt}\rVert_{L^2}^2 \right)&\le -C \left(\rVert \partial_1\Delta\theta\rVert_{L^2}^2 + \rVert u\rVert_{H^5}^2+\rVert w\rVert_{H^5}^2 \right) + C(A_1+A_2),
\end{align}
where
\begin{equation}
\begin{aligned}\label{A123def}
A_1&:= \left( \rVert u_2\rVert_{W^{2,\infty}} + \rVert u_2\rVert_{H^3} +  \rVert \Delta^2w \rVert_{L^2}\right)\rVert \Delta^2\theta\rVert_{L^2}\\
A_2&:=\rVert \partial_1\theta \rVert_{L^2}^2 +\rVert u\rVert_{H^1}^2 +\rVert w_{tt}\rVert_{L^2}^2+ \rVert w\rVert_{H^2}^2.
\end{aligned}
\end{equation}
\end{proposition}
 

\begin{proof} Let us compute the evolution of $\rVert \Delta^2\theta\rVert_{L^2}^2, \rVert w_t\rVert_{L^2}^2,\rVert w_{tt}\rVert_{L^2}^2$ separately.

 \textbf{Computation for $\rVert \Delta^2\theta\rVert_{L^2}$.}
 From \eqref{Boussiesq_pertur}, we have
\begin{align}\label{whyogfots2}
\frac{1}2\frac{d}{dt}\rVert \Delta^2 \theta\rVert_{L^2}^2 = -\int \Delta^2 (u\cdot\nabla \theta)\Delta^2 \theta dx +\int \Delta^2(-\partial_2\rho_s u_2)\Delta^2 \theta dx
\end{align}
We simplify the linear term first. We rewrite it as
\begin{align}\label{dissipative_term_Boussines}
\int \Delta^2 (-\partial_2\rho_s u_2)\Delta^2 \theta dx&=\int \left(\Delta^2 (-\partial_2\rho_s u_2) -(-\partial_2\rho_s\Delta^2 u_2)\right) \Delta^2\theta dx + \int -\partial_2\rho_s \Delta^2 u_2 \Delta^2 \theta dx =: I_1+ I_2,
\end{align}
while the second integral can be further decomposed into
\begin{align}\label{dissipative_term_Boussines11232}
I_2 = \int -\partial_2\rho_s \Delta^2 (u_2-v_2) \Delta^2 \theta dx +\int -\partial_2\rho_s \Delta^2 v_2 \Delta^2 \theta dx=: I_{21}+I_{22}.
\end{align}
Since $\rho_s$ is independent of $x_1$, the  the stream function of $v$ in \eqref{streamv2} can be restated as
 \begin{align}\label{stream22v2}
  \begin{cases}
  \Delta^2 \Psi = \partial_1 \theta & \text{ in $\Omega$,}\\
  \Psi=\nabla \Psi = 0  & \text{ on $\partial\Omega$.}
  \end{cases}
  \end{align} Thus we can rewrite $I_{22}$ as
\begin{align}\label{defiu1sdsxc2}
I_{22} &= \int -\partial_2\rho_s \partial_1\Delta^2\Psi \Delta^2\theta dx = \int \partial_2\rho_s\partial_1\theta \partial_1\Delta^2\theta dx =  \int \Delta (\partial_2\rho_s \partial_1\theta) \partial_1\Delta\theta dx\nonumber\\
& = \int \partial_2\rho_s |\partial_1\Delta \theta|^2 dx + \int (\partial_{222}\rho_s\partial_1\theta + 2\partial_{22}\rho_s\partial_{12}\theta)\partial_1\Delta\theta dx\nonumber\\
& =: I_{221}+I_{222},
\end{align}
where the integration by parts in the third equality is justified by the boundary condition of $\theta$ in Lemma~\ref{vanishing_boundary_boy}. Hence, it follows from \eqref{dissipative_term_Boussines}, \eqref{dissipative_term_Boussines11232} and \eqref{defiu1sdsxc2} that the linear term can be decomposed as
\begin{align}\label{rjworstlex}
\int \Delta^2 (-\partial_2\rho_s u_2)\Delta^2 \theta dx - I_{221} = I_1 +I_{21} + I_{222}.
\end{align}
For $I_1$ in \eqref{dissipative_term_Boussines}, we estimate
\[
|I_1|\le  \left(\rVert\partial_2\rho_s\rVert_{H^4}\rVert u_2\rVert_{L^\infty} + \rVert \partial_2\rho_s\rVert_{L^\infty}\rVert u_2\rVert_{H^3}\right)\rVert \Delta^2\theta\rVert_{L^2}\le C\rVert u_2\rVert_{H^3}\rVert \Delta^2\theta\rVert_{L^2},
\]
where the last inequality follows from \eqref{steadtysd2sd} and \eqref{embd1bo}.
For $I_{21}$,  we apply the Cauchy-Schwarz inequality and the definition of $w$ to obtain
\[
|I_{21}|\le \rVert \partial_2\rho_s\rVert_{L^\infty}\rVert w\rVert_{H^4}\rVert \Delta^2\theta\rVert_{L^2}\le C\rVert w\rVert_{H^4}\rVert \Delta^2\theta\rVert_{L^2}.
\]
For $I_{222}$, we have
\[
|I_{222}|\le  \rVert \partial_{222}\rho_s\rVert_{L^\infty}\rVert \partial_1\theta\rVert_{H^1}\rVert \partial_1\Delta\theta\rVert_{L^2}\le C\rVert \partial_1\theta \rVert_{H^1}\rVert \partial_1\Delta\theta\rVert_{L^2}.
\]
Using the interpolation inequality, we have $\rVert \partial_1\theta\rVert_{H^1}\le C \rVert \partial_1\theta\rVert_{L^2}^{1/2}\rVert \partial_1\theta\rVert_{H^2}^{1/2}\le \rVert \partial_1\theta\rVert_{L^2}^{1/2}\rVert \partial_1\Delta\theta\rVert_{L^2}^{1/2}$, where the last inequality is due to Lemma~\ref{bilap}. Hence we have
\[
|I_{222}|\le C \rVert \partial_1\theta\rVert_{L^2}^{1/2}\rVert \partial_1\Delta\theta\rVert_{H^2}^{3/2}\le  C_\eta\rVert \partial_1\theta\rVert_{L^2}^{2}+\eta \rVert \partial_1\theta\rVert_{H^2}^{2},\text{ for any $\eta>0$,}
\]
where the last inequality follows from  Young's inequality. Collecting the above estimates, we arrive at
\begin{align}\label{wanja1}
|I_{1}|+|I_{21}|+|I_{222}|\le C(\rVert u_2\rVert_{H^3} + \rVert w\rVert_{H^4})\rVert \Delta^2\theta\rVert_{L^2} + C_\eta\rVert \partial_1\theta \rVert_{L^2}^2 + \eta\rVert \partial_1\Delta\theta\rVert_{L^2}^2, \text{ for any $\eta>0$}.
\end{align}
Note that $I_{221}\le -\gamma \rVert \partial_1\Delta\theta\rVert_{L^2}^2\le_C - \rVert \partial_1\Delta\theta\rVert_{L^2}^2$ which follows from the upper bound of $\partial_2\rho_s$ in \eqref{steadtysd2sd} and the definition of $I_{221}$ in \eqref{defiu1sdsxc2}. Hence, we can choose $\eta>0$ in \eqref{wanja1} small enough so that  \eqref{rjworstlex} reads
\begin{align}\label{rjworstlexdsd2deng}
\int \Delta^2 (-\partial_2\rho_s u_2)\Delta^2 \theta dx\le -C\rVert \partial_1\Delta\theta\rVert_{L^2}^2 + C(\rVert u_2\rVert_{H^3} + \rVert w\rVert_{H^4})\rVert \Delta^2\theta\rVert_{L^2} + C\rVert \partial_1\theta \rVert_{L^2}^2.
\end{align} 
This gives an estimate for the second integral in the right-hand side of \eqref{whyogfots2}. 

In order to estimate the nonlinear term $\int \Delta^2 (u\cdot\nabla \theta)\Delta^2 \theta dx$ in \eqref{whyogfots2}, we recall the following Lemma:
 \begin{lemma}\label{energy_lemma_viscous}\cite[Lemma 4.5]{park2024stability}
Let $u$ be a smooth divergence free vector field such that $U=0$ on $\partial \Omega$ and $f$ be a smooth scalar-valued function such that $f=\partial_2 f=\partial_{2}^2 \overline{f}=0$ on $\partial \Omega$ where $\overline{f}:=\frac{1}{2\pi}\int_{\mathbb{T}}f(x_1,x_2)dx_1$. Then, we have
\begin{align*}
&\left| \int_{\Omega}\Delta^2(u\cdot\nabla f)\Delta^2f dx\right|\le_C\left(\rVert U_2\rVert_{H^3(\Omega)} +  \rVert U_2\rVert_{W^{2,\infty}(\Omega)}\right)\rVert \Delta^2f\rVert_{L^2(\Omega)}^2 +\rVert U\rVert_{H^5(\Omega)}\rVert \partial_1\Delta f\rVert_{L^2(\Omega)}\rVert \Delta^2 f\rVert_{L^2(\Omega)},
\end{align*}
where $C>0$ is a universal constant.
\end{lemma}

Thanks to Lemma~\ref{vanishing_boundary_boy} and the boundary condition of $u$ in \eqref{bdcondition}, we can apply the above lemma with $U=u$ and $f=\theta$ which immediately gives us
\begin{align}\label{n1first}
\left|\int \Delta^2 (u\cdot\nabla \theta)\Delta^2 \theta dx\right|&\le_C  \left(\rVert u_2\rVert_{H^3} +  \rVert u_2\rVert_{W^{2,\infty}}\right)\rVert \Delta^2\theta\rVert_{L^2}^2 +\rVert u\rVert_{H^5}\rVert \partial_1\Delta \theta\rVert_{L^2}\rVert \Delta^2 \theta\rVert_{L^2}.
\end{align}
To control the highest order term, $\rVert u\rVert_{H^5}$, we rewrite it as $u=w+v$, thus
\[
\rVert u\rVert_{H^5}\le \rVert w\rVert_{H^5} + \rVert v\rVert_{H^5}.
\]
Applying Lemma~\ref{lemma_psi_stokes} to \eqref{stream22v2}, we have
\begin{align}\label{Wanjaplease}
\rVert v\rVert_{H^5}\le \rVert \Phi\rVert_{H^6}\le C\rVert \partial_1\theta\rVert_{H^2}\le C\rVert \partial_1\Delta \theta\rVert_{L^2},
\end{align}
where the last inequality is due to Lemma~\ref{bilap}. With these estimates, the last term in \eqref{n1first} can be bounded as
\begin{align*}
\rVert u\rVert_{H^5}\rVert \partial_1\Delta \theta\rVert_{L^2}\rVert \Delta^2 \theta\rVert_{L^2}&\le \rVert w \rVert_{H^5}\rVert \partial_1\Delta \theta\rVert_{L^2}\rVert \Delta^2 \theta\rVert_{L^2} + C\rVert \partial_1\Delta \theta\rVert_{L^2}^2\rVert \Delta^2 \theta\rVert_{L^2} \\
&\le \rVert w \rVert_{H^5}^2\rVert \Delta^2\theta\rVert_{L^2} + C\rVert \partial_1\Delta \theta\rVert_{L^2}^2\rVert \Delta^2 \theta\rVert_{L^2} ,
\end{align*}
where we used Young's inequality for the last inequality. Plugging this into \eqref{n1first}, we get
\[
\left|\int \Delta^2 (u\cdot\nabla \theta)\Delta^2 \theta dx\right| \le_C   \left(\rVert u_2\rVert_{H^3} +  \rVert u_2\rVert_{W^{2,\infty}}\right)\rVert \Delta^2\theta\rVert_{L^2}^2  + \rVert w \rVert_{H^5}^2\rVert \Delta^2\theta\rVert_{L^2} + \rVert \partial_1\Delta \theta\rVert_{L^2}^2\rVert \Delta^2 \theta\rVert_{L^2}.
\] Plugging this and \eqref{rjworstlexdsd2deng} into \eqref{whyogfots2}, we arrive at
\begin{align*}
\frac{1}2\frac{d}{dt}\rVert \Delta^2 \theta\rVert_{L^2}^2&\le -C\left(1-C\rVert \Delta^2\theta\rVert_{L^2}\right)\rVert \partial_1\Delta\theta\rVert_{L^2}^2 \nonumber\\
& \ +  C\left(\rVert u_2\rVert_{H^3} +  \rVert u_2\rVert_{W^{2,\infty}}\right)\rVert \Delta^2\theta\rVert_{L^2}^2 +C\left(\rVert u_2\rVert_{H^3} +  \rVert \Delta^2w \rVert_{L^2}\right)\rVert \Delta^2\theta\rVert_{L^2}\nonumber \\
& \ +C\rVert w \rVert_{H^5}^2\rVert \Delta^2\theta\rVert_{L^2} +C\rVert \partial_1\theta \rVert_{L^2}^2.
\end{align*}
 Using \eqref{size_assu}, we make this slightly simpler, that is,
\begin{align}\label{crazybitch24r2sd}
\frac{1}2\frac{d}{dt}\rVert \Delta^2 \theta\rVert_{L^2}^2&\le -C\rVert \partial_1\Delta\theta\rVert_{L^2}^2 \nonumber\\
& \ + C\left( \rVert u_2\rVert_{W^{2,\infty}} + \rVert u_2\rVert_{H^3} +  \rVert \Delta^2w \rVert_{L^2}\right)\rVert \Delta^2\theta\rVert_{L^2}\nonumber \\
& \ +C\rVert w \rVert_{H^5}^2\rVert \Delta^2\theta\rVert_{L^2} +C\rVert \partial_1\theta \rVert_{L^2}^2.
\end{align}

 \textbf{Computation for $\rVert w_t\rVert_{L^2}$.}
 Now we turn to derive an estimate for $\rVert w_{t}\rVert_{L^2}^2$.  Using \eqref{wtt_1}, we compute
\begin{align}\label{wnawjasd1trhsd}
\frac{1}2\frac{d}{dt}\rVert w_{t}\rVert_{L^2}^2 &= -\int F_{t} w_{t} dx - \int |\nabla w_{t}|^2dx\le C\rVert F_{t}\rVert_{L^2}\rVert w_{t}\rVert_{L^2} - \rVert \nabla w_{t}\rVert_{L^2}^2\nonumber\\
&\le_C \rVert F_{t}\rVert_{L^2}^2 - C\rVert \nabla w_{t}\rVert_{L^2}^2\nonumber\\
&\le_C \left(\rVert w\rVert_{H^2}^2+\rVert u\rVert_{H^1}^2 + \rVert w_{tt}\rVert_{L^2}^2  \right) - C\rVert \nabla w_{t}\rVert_{L^2}^2,
\end{align}
where the second last inequality is due to Young's inequality and the last inequality is from \eqref{BouFestimate3} and \eqref{size_assu}.

 \textbf{Computation for $\rVert w_{tt}\rVert_{L^2}$.}
 Next step, we  estimate for $\rVert w_{tt}\rVert_{L^2}^2$ in a similar way. Using \eqref{wtt_2}, we compute
\begin{align*}
\frac{1}2\frac{d}{dt}\rVert w_{tt}\rVert_{L^2}^2 &= -\int F_{tt} w_{tt} dx - \int |\nabla w_{tt}|^2dx\le C\rVert F_{tt}\rVert_{H^{-1}}\rVert\nabla w_{tt}\rVert_{L^2} -\rVert \nabla w_{tt}\rVert_{L^2}^2\\
&\le_C \rVert F_{tt}\rVert_{H^{-1}}^2 - C\rVert \nabla w_{tt}\rVert_{L^2}^2,
\end{align*}
where the first inequality is due to the Poincar\'e inequality, $\rVert w_{tt}\rVert_{H^1}\le_C\rVert \nabla w_{tt}\rVert_{L^2}$ and the last inequality is due to Young's inequality. Using \eqref{BouFestimate5}, we arrive at
\begin{align}\label{arriv21sd}
\frac{1}2\frac{d}{dt}\rVert w_{tt}\rVert_{L^2}^2 \le_C  \rVert w_{tt}\rVert_{L^2}^2+ \rVert w\rVert_{H^2}^2+\rVert u\rVert_{H^1}^2 - C\rVert \nabla w_{tt}\rVert_{L^2}^2.
\end{align}
Again using $\rVert w_{tt}\rVert_{H^1}\le C\rVert \nabla w_{tt}\rVert_{L^2}$ and combining this with \eqref{leaving3}, we get
\begin{align*}
\rVert \nabla w_{tt}\rVert_{L^2}^2&\ge C\rVert w\rVert_{H^5}^2 - C\left(\rVert F\rVert_{H^3}^2 + \rVert F_t\rVert_{H^1}^2\right) \\
&\ge C\rVert w\rVert_{H^5} - C\left( \rVert u\rVert_{H^4}^2 \rVert w_{tt}\rVert_{L^2}^2 + \rVert w\rVert_{H^2}^2+\rVert u\rVert_{H^3}^2\right)\\
&\ge C\rVert w\rVert_{H^5} - C\left(\rVert w_{tt}\rVert_{L^2}^2 + \rVert w\rVert_{H^2}^2+\rVert u\rVert_{H^3}^2\right)
\end{align*}
where the second last inequality is due to \eqref{BouFestimate2} and \eqref{BouFestimate4}, and the last inequality is due to  \eqref{size_assu}. Thus, from \eqref{arriv21sd}, we arrive that
\begin{align}\label{arri2223123v21sd}
\frac{1}2\frac{d}{dt}\rVert w_{tt}\rVert_{L^2}^2 \le_C  \rVert w_{tt}\rVert_{L^2}^2+ \rVert w\rVert_{H^2}^2+\rVert u\rVert_{H^3}^2 - C\rVert  w\rVert_{H^5}^2.
\end{align}

Now, we combine all the estimates   \eqref{crazybitch24r2sd}, \eqref{wnawjasd1trhsd} and \eqref{arri2223123v21sd}, we arrive at
\begin{align}
\frac{1}{2}\frac{d}{dt}\left(\rVert \Delta^2\theta\rVert_{L^2}^2+\rVert w_t\rVert_{L^2}^2+\rVert w_{tt}\rVert_{L^2}^2 \right)&\le \underbrace{-C \left(\rVert \partial_1\Delta\theta\rVert_{L^2}^2 +\rVert w\rVert_{H^5}^2 \right) +  C\rVert w \rVert_{H^5}^2\rVert \Delta^2\theta\rVert_{L^2} + \rVert u\rVert_{H^3}^2}_{=:J_1}\label{rj2sdpposdsd}\\
& \  + C\left( \rVert u_2\rVert_{W^{2,\infty}} + \rVert u_2\rVert_{H^3} +  \rVert \Delta^2w \rVert_{L^2}\right)\rVert \Delta^2\theta\rVert_{L^2}\nonumber\\
& \ +C\left(\rVert \partial_1\theta \rVert_{L^2}^2  +\rVert w_{tt}\rVert_{L^2}^2+ \rVert w\rVert_{H^2}^2 \right).\nonumber
\end{align}

 Let us simplify the terms $J_1$.  From \eqref{size_assu}, we can assume $\rVert \Delta^2 \theta\rVert_{L^2}$ is sufficiently small so that
 \begin{align}\label{wanja2}
 J_1 \le  -C\left(\rVert \partial_1\Delta\theta\rVert_{L^2}^2 +\rVert w\rVert_{H^5}^2 \right)  + \rVert u\rVert_{H^3}^2.
 \end{align}
 Using $w=u-v$, we also have
  \[
 \rVert u\rVert_{H^5}\le_C \rVert w\rVert_{H^5} + \rVert v\rVert_{H^5}\le_C \rVert w\rVert_{H^5}+\rVert \Psi\rVert_{H^6}\le_C \rVert w\rVert_{H^5}  + \rVert \partial_1\theta\rVert_{H^2}\le_C\rVert w\rVert_{H^5}  + \rVert \partial_1\Delta\theta\rVert_{L^2},
 \]
 where the second and the third inequality are due to \eqref{stream22v2} and Lemma~\ref{lemma_psi_stokes}, and the last inequality follows from Lemma~\ref{bilap}. With this, we rewrite the estimate \eqref{wanja2} as
\[
J_1 \le  -C\left(\rVert \partial_1\Delta\theta\rVert_{L^2}^2 +\rVert w\rVert_{H^5}^2 +\rVert u\rVert_{H^5} \right)  + \rVert u\rVert_{H^3}^2.
\]
 Also,  the Gagliardo-Nirenberg interpolation theorem tells us
\[
\rVert u\rVert_{H^3}^2\le_C \rVert u\rVert_{H^1}\rVert u\rVert_{H^5}\le_C \eta \rVert u\rVert_{H^5}^{2} + C_\eta\rVert u\rVert_{H^1}^2,\text{ for any $\eta>0$,}
\]
where the last inequality is due to Young's inequality. Thus by choosing $\eta$ sufficiently small, we can estimate $J_1$  as
\[
J_1 \le -C \left(\rVert \partial_1\Delta\theta\rVert_{L^2}^2 + \rVert w \rVert_{H^5}^2 + \rVert u\rVert_{H^5}\right) + C\rVert u\rVert_{H^1}^2.
\]
Hence plugging this into \eqref{rj2sdpposdsd}, we obtain the desired estimate.
\end{proof}

\subsection{Derivation of decay rates via  energy structure} In this subsection, the main goal is to derive sufficient decay rates in time for the terms $A_1$ and $A_2$ in \eqref{rakwk123bo}, under the assumption \eqref{size_assu}. We will achieve this as an application of the Lyapunov functional constructed in Section~\ref{Lyapunsec}. As a first step, we derive an explicit decay rate of the total energy and the enstrophy (in a time average sense).   
 
  As a preparation, let us  recall the following lemma, showing that the vertical rearrangement is indeed a minimizer of the potential energy:
  \begin{lemma}\cite[Proposition 2.5]{park2024stability}\label{propoos}
  Suppose $\rho_s$ satisfies \eqref{steadtysd2sd}. There exists $\delta=\delta(\gamma,\rVert \rho_s\rVert_{H^4})>0$ such that if $f={\rho}_s$ on $\partial\Omega$ and  $\rVert f-{\rho}_s\rVert_{H^3(\Omega)}\le \delta$, then
 \begin{align}\label{energy_nondegeneracy}
 C^{-1} \rVert f-f^*\rVert_{L^2(\Omega)}^2\le \int_{\Omega}(f(x)-f^*(x))x_2dx\le C\rVert f-f^*\rVert_{L^2(\Omega)}^2,
 \end{align}
 where $f^*$ is the vertical rearrangement of $f$, defined in \eqref{verticla_re}. 
 Moreover, we have
 \begin{align}\label{hatedogs}
 \rVert \partial_1f \rVert_{L^2(\Omega)}\ge C\rVert f-f^*\rVert_{L^2(\Omega)}.
 \end{align}
 The constant $C$ depends  only on $\gamma$ and $\rVert \rho_s\rVert_{H^4}$.  \end{lemma}

\begin{proposition}\label{|energyedeasd}
Let $(\rho(t),u(t))$ be a solution to the Boussinesq system \eqref{Bouss} for $t\in [0,T]$ for some $T>0$. Then there exists a small constant  $\delta_0=\delta_0(\rho_s)>0$ such that if \eqref{size_assu} holds, then
 \begin{align}
 E_T(t) +S(t)&\le C\frac{\delta}{t^2},\label{energy_decay_bou}\\
 \frac{2}{t}\int_{t/2}^t\rVert \nabla u(s)\rVert_{L^2}^2 +\rVert \nabla v\rVert_{L^2}^2 + \rVert \nabla w(s)\rVert_{L^2}^2ds& \le C\frac{\delta}{t^3}.\label{energy_decay_bou_2}
\end{align}
\end{proposition}
\begin{proof}

Assuming \eqref{size_assu} for sufficiently small $\delta_0\ll 1$, it is trivial that $\rVert \rho_0\rVert_{L^4}$ and $E_T(0)$ is bounded by a universal constant.  Therefore Proposition~\ref{lyaypiunove} tells us that
\begin{align}\label{riosxopsdw2}
\frac{d}{dt}\left(CE_T(t) + S(t) \right)\le -C\left(\rVert \nabla w\rVert_{L^2}^2+ \rVert \nabla v\rVert_{L^2}^2 + \rVert \nabla u\rVert_{L^2}^2\right)\le -C\left(\rVert u\rVert_{L^2}^2 + \rVert w\rVert_{L^2}^2 + \rVert \nabla v\rVert_{L^2}^2 \right),
\end{align}
where the last inequality follows from the Poincar\'e inequality.
 
 Now, let us estimate the lower bound of $\rVert \nabla v\rVert_{L^2}^2.$ From the stream function formulation in \eqref{streamv2}, we have
\begin{align}\label{lower_nab}
\rVert \nabla v\rVert_{L^2}^2 = \rVert \Phi\rVert_{H^2}^2.
\end{align}
Note that \eqref{streamv2} also implies
\begin{align}\label{misswanja12}
\rVert \partial_1\rho \rVert_{L^2} = \rVert \Delta^2 \Phi\rVert_{L^2}^2\le_C \rVert \Delta^2 \Phi\rVert_{H^2}^{1/2}\rVert \Delta \Phi\rVert_{L^2}^{1/2},
\end{align}
where the last inequality is due to the Gagliardo-Nirenberg interpolation inequality. Also, Lemma~\ref{propoos}, under the assumption \eqref{size_assu} for sufficiently small $\delta_0$, implies
\begin{align}\label{energy_compatibility}
C^{-1}\rVert \rho(t)-\rho_0^*\rVert_{L^2}^2 \le E_P(t)\le C\rVert \rho(t)-\rho^*\rVert_{L^2}^2\le C \rVert \partial_1 \rho\rVert_{L^2}^2.
\end{align}
Hence,  combining this estimate and \eqref{misswanja12}, we obtain
\[
\rVert \Phi\rVert_{H^2}^2 \ge \rVert \Delta \Phi\rVert_{L^2}^2 \ge_C \rVert \partial_1\rho\rVert_{L^2}^4 \rVert \Delta^2\Phi\rVert_{H^2}^{-2}\ge_C \rVert \rho-\rho^*\rVert_{L^2}^4\rVert \Delta^2\Phi\rVert_{H^2}^{-2}\ge_C  E_P(t)^2 \rVert \Delta^2\Phi\rVert_{H^2}^{-2}.
\]
Plugging this into \eqref{lower_nab}, we arrive at
\[
\rVert \nabla v\rVert_{L^2}^2 \ge_C E_P(t)^2 \rVert \Delta^2\Phi\rVert_{H^2}^{-2}.
\]
This implies that the energy estimate in \eqref{riosxopsdw2} can be written as 
\[
\frac{d}{dt}\left( CE_P(t) + CE_K(t) + S(t) \right)\le_C -\left((CE_P(t))^2\rVert \Delta^2\Phi(t)\rVert_{H^2}^{-2} + CE_K(t)+S(t) \right).
\]
Then applying Lemma~\ref{odelem} with $f(t):=CE_P(t)$, $g(t):= CE_K(t)+S(t)$ and $\alpha(t):= \rVert \Delta^2\Phi(t)\rVert_{H^2}^{2}$, we obtain
\begin{align}\label{morecarefu1}
 CE_P(t) + CE_K(t) + S(t) \le C\frac{1}{t^2}\left(\int_0^T \rVert \Delta^2 \Phi\rVert_{H^2}^2dt + \left(E_P(0)+E_K(0)+S(0)\right)\right)\text{ for all $t\in [0,T]$}.
\end{align}
On the other hand, \eqref{streamv2}, \eqref{size_assu} and Lemma~\ref{bilap} implies
\begin{align}\label{gogl3sd}
\int_0^T \rVert \Delta^2 \Phi\rVert_{H^2}^2dt\le_C \int_0^T \rVert \partial_1\theta\rVert_{H^2}^2 dt \le_C\delta.
\end{align}
From \eqref{energy_compatibility} and \eqref{lower_nab}, we have
\begin{align}\label{daysiof2sd}
E_P(t)+ \rVert \nabla v\rVert_{L^2}^2 \le_C \rVert \partial_1\rho\rVert_{L^2}^2 + \rVert \Phi\rVert_{H^2}^2 \le_C \rVert \partial_1\theta\rVert_{L^2}^2\le_C \delta,
\end{align}
where the second last inequality is due to \eqref{streamv2} and Lemma~\ref{lemma_psi_stokes} and the last inequality is due to \eqref{size_assu}. Moreover, we have
\[
E_K(t)+S(t)\le_C \rVert u(t)\rVert_{L^2}^2 + \rVert w(t)\rVert_{L^2}^2 \le_C \rVert u\rVert_{L^2}^2 + \rVert v\rVert_{L^2}^2\le_C \delta,
\]
were the last inequality is from \eqref{size_assu} and \eqref{daysiof2sd}. Hence, plugging this, \eqref{daysiof2sd} and \eqref{gogl3sd} into \eqref{morecarefu1}, we obtain  \eqref{energy_decay_bou}.
 
Lastly,  integrating \eqref{riosxopsdw2} in time over $[t/2,t]$, we get
\[
CE_T(t)+S(t) +C\int_{t/2}^t \rVert \nabla u\rVert_{L^2} + \rVert \nabla v\rVert_{L^2}^2 +\rVert \nabla w\rVert_{L^2}^2 ds \le C E_T(t/2) + S(t/2)\le \frac{C\delta}{t^2},
\]
where the last inequality is due to \eqref{energy_decay_bou}. Since $E_T(t), S(t)\ge0$, 
this gives us \eqref{energy_decay_bou_2}.
\end{proof}

In view of Proposition~\ref{bouenergy_es}, the decay rates obtained in the above proposition is not sufficient, since they cannot control high Sobolev norms such as  $\rVert u_2\rVert_{H^3},  \rVert u_2\rVert_{W^{2,\infty}}, \rVert \Delta^2w \rVert_{L^2}$ and $\rVert w_{tt}\rVert_{L^2}$.  Sufficient decay rates for these high norms will be achieved in the next proposition.

\begin{proposition}\label{higherdrhappy}
Let $(\rho(t),u(t))$ be a solution to the Boussinesq system \eqref{Bouss}. Then there exists a small constant  $\delta_0=\delta_0(\rho_s)>0$ such that if \eqref{size_assu} holds for $\delta<\delta_0$, then
\begin{align*}
\frac{2}{t}\int_{t/2}^t \rVert \partial_1\theta(s)\rVert_{L^2}^2ds&\le_C \frac{\delta}{t^2},\\
\frac{2}{t}\int_{t/2}^t \rVert w_{tt}(s)\rVert_{L^2}^2+\rVert w(s)\rVert_{H^4}^2 + \rVert u_2(s)\rVert_{H^2}^2ds &\le_C \frac{\delta}{t^3},
\end{align*}
 \text{ for all $t\in [0,T]$.}
\end{proposition}
\begin{proof}
\textbf{Estimate for $\rVert w_{tt}\rVert_{L^2}$.}
 Taking the Leray projection $\mathbb{P}$ in \eqref{evol_w} and differentiating it, we observe
\[
\nabla w_t = -\nabla \mathbb{P}  F + \nabla \mathbb{P}\Delta w.
\]
A standard energy estimate gives us
\[
\frac{d}{dt}\rVert \nabla w \rVert_{L^2}^2 = - \int \nabla \mathbb{P} F \nabla wdx + \int \nabla \mathbb{P}\Delta w \nabla w dx = \int F \mathbb{P} \Delta w dx - \int |\mathbb{P}\Delta w|^2 dx\le -C\rVert \mathbb{P}\Delta w \rVert_{L^2}^2 + C\rVert F\rVert_{L^2}^2,
\]
where the last inequality follows from Young's inequality. Using \eqref{BouFestimate11}, we arrive at
\[
\frac{d}{dt}\rVert \nabla w \rVert_{L^2}^2  \le -C \rVert \mathbb{P}\Delta w \rVert_{L^2}^2 + C\rVert \nabla u\rVert_{L^2}^2.
\]
Then, applying Lemma~\ref{odeep21} to the above differential inequality, noting the estimates for $\nabla u$ and $\nabla w$ in  \eqref{energy_decay_bou_2}, we obtain
\begin{align}\label{yesterday}
\frac{2}{t}\int_{t/2}^t\rVert \mathbb{P}\Delta w \rVert_{L^2}^2 ds \le_C \frac{\delta}{t^3}.
\end{align}
Again, applying the Leray projection to \eqref{evol_w}, we  have $
w_t = -\mathbb{P}F + \mathbb{P}\Delta w,$
hence
\[
\rVert w_t \rVert_{L^2}\le_C \rVert F\rVert_{L^2} + \rVert \mathbb{P}\Delta w\rVert_{L^2}.
\]
Applying Lemma~\ref{Stokes_wellposed} to \eqref{evol_w}, we also have
\[
\rVert w\rVert_{H^2}\le \rVert w_t\rVert_{L^2}+ \rVert F\rVert_{L^2}\le_C \rVert F\rVert_{L^2} + \rVert \mathbb{P}\Delta w\rVert_{L^2}\le_C \rVert \nabla u\rVert_{L^2} + \rVert \mathbb{P}\Delta w\rVert_{L^2},
\]
where the last inequality follows from \eqref{BouFestimate11}.
Hence, using \eqref{yesterday} and \eqref{energy_decay_bou_2}, we arrive at
\begin{align}\label{thats}
\frac{2}{t}\int_{t/2}^t\rVert  w \rVert_{H^2}^2 ds\le \frac{2}{t}\int_{t/2}^t \rVert \nabla u\rVert_{L^2}^2 + \rVert \mathbb{P}\Delta w\rVert_{L^2}^2 ds \le_C  \frac{\delta}{t^3}.
\end{align}
Towards the higher derivative estimate for $w$, we take the Leray projection $\mathbb{P}$ in the equation of $w_{tt}$ in \eqref{wtt_1}, yielding that 
\begin{align}\label{rjsdroisskwj2sds}
w_{tt}= -\mathbb{P}F_t+\mathbb{P}\Delta w_{t}.
\end{align}
A standard energy estimate gives us
\begin{align}\label{r22sdoss2sd}
\frac{d}{dt}\frac{1}{2}\rVert \nabla w_{t}\rVert_{L^2}^2 & = -\int \nabla\mathbb{P}(F_t) \nabla w_t dx +\int  \nabla w_t\nabla\mathbb{P} \Delta w_t dx = -\int F_t \mathbb{P}\Delta w_t dx - \int |\mathbb{P}\Delta w_t|^2 dx\nonumber\\
&\le C\rVert F_t\rVert_{L^2}^2 -\frac{1}{2}\rVert \mathbb{P}\Delta w_t\rVert_{L^2}^2\nonumber\\
&\le C\rVert F_t\rVert_{L^2}^2 -C\rVert \mathbb{P}\Delta w_t\rVert_{L^2}^2 - C\rVert w_t\rVert_{H^2}^2\nonumber\\
&\le C\rVert F_t\rVert_{L^2}^2 -C\rVert \mathbb{P}\Delta w_t\rVert_{L^2}^2 - C\rVert \nabla w_t\rVert_{L^2}^2,
\end{align}
where the second  inequality follows from \eqref{rkawk1ppoross}.
From the boundedness of the Leray projection (see Lemma~\ref{leray_pro}), we notice that \eqref{rjsdroisskwj2sds} implies 
\[
\rVert \mathbb{P}\Delta w_t\rVert_{L^2}\ge \rVert w_{tt}\rVert_{L^2} - \rVert F_t\rVert_{L^2}.
\]
Thus \eqref{r22sdoss2sd} yields
\begin{align}\label{getupwanja2}
\frac{d}{dt}\frac{1}{2}\rVert \nabla w_{t}\rVert_{L^2}^2  \le -C\rVert \nabla w_t\rVert_{L^2}^2 -C \rVert w_{tt}\rVert_{L^2}^2 + C\rVert F_t\rVert_{L^2}^2
\end{align}
Using \eqref{BouFestimate3} we see that 
\[
-C \rVert w_{tt}\rVert_{L^2} + C\rVert F_t\rVert_{L^2}^2\le -C(1-C\rVert u\rVert_{H^3})\rVert w_{tt}\rVert_{L^2} +C\left(\rVert w\rVert_{H^2}+\rVert u\rVert_{H^1}\right)\le - C \rVert w_{tt}\rVert_{L^2} + C\left(\rVert w\rVert_{H^2}+\rVert u\rVert_{H^1} \right),
\]
where the last inequality is due to \eqref{size_assu}. Plugging this into \eqref{getupwanja2}, we arrive at 
\begin{align}\label{rjsowhjxxsd23sd}
\frac{d}{dt}\frac{1}{2}\rVert \nabla w_{t}\rVert_{L^2}^2  \le -C\rVert \nabla w_t\rVert_{L^2}^2 - C \rVert w_{tt}\rVert_{L^2}^2 + C\left(\rVert w\rVert_{H^2}^2+\rVert u\rVert_{H^1}^2 \right).
\end{align}
We notice that the last term in the right-hand side satisfies $\frac{2}{t}\int_{t/2}^{t} \rVert w(s)\rVert_{H^2}^2+\rVert u(s)\rVert_{H^1}^2ds \le_C \frac{\delta}{t^3}$, which is due to \eqref{thats} and \eqref{energy_decay_bou_2}. Thus applying Lemma~\ref{odee2p21} to \eqref{rjsowhjxxsd23sd}, we obtain
 \begin{align}\label{nalsd2}
 \frac{2}t\int_{t/2}^t \rVert \nabla w_t \rVert_{L^2}^2 ds \le \frac{C}{t^3} \left(\rVert \nabla w_t(0)\rVert_{L^2}^2 + \int_0^T \rVert w(s)\rVert_{H^2}^2+\rVert u(s)\rVert_{H^1}^2 ds\right).
 \end{align}
Applying Lemma~\ref{Stokes_wellposed} to \eqref{wtt_1} and \eqref{wtt_2},  we see that
\begin{align}\label{rkjsdj2xc211sd}
\rVert \nabla w_t\rVert_{L^2} \le_C \rVert w_t\rVert_{H^2}\le_C \rVert w_{tt}\rVert_{L^2} + \rVert F_t\rVert_{L^2},\quad \rVert w\rVert_{H^2}\le_C \rVert F\rVert_{L^2} + \rVert w_t\rVert_{L^2}.
\end{align}
Hence,  \eqref{BouFestimate3} tell us that
\begin{align*}
\rVert \nabla w_t(t)\rVert_{L^2} &\le_C (1+\rVert u\rVert_{H^3}) \rVert w_{tt}\rVert_{L^2} + \rVert w\rVert_{H^2} + \rVert u\rVert_{H^1}\\
&\le_C  (1+\rVert u\rVert_{H^3}) \rVert w_{tt}\rVert_{L^2}  + \rVert F\rVert_{L^2} + \rVert w_t\rVert_{L^2} +\rVert u\rVert_{H^1}\\
&\le_C  (1+\rVert u\rVert_{H^3}) \rVert w_{tt}\rVert_{L^2} + \rVert w_t\rVert_{L^2} + \rVert u\rVert_{H^1}\\
&\le_C \sqrt{\delta},
\end{align*}
where the second inequality is due to the second inequality in \eqref{rkjsdj2xc211sd}, and the third inequality follows from \eqref{BouFestimate11}, and the last inequality is due to \eqref{size_assu}. Together with \eqref{size_assu}, this implies that
\[
\rVert \nabla w_t(0)\rVert_{L^2}^2 + \int_0^T \rVert w(s)\rVert_{H^2}^2+\rVert u(s)\rVert_{H^1}^2 ds\le_C \delta.
\]
Thus, \eqref{nalsd2} reads as
\begin{align}\label{read1sds}
 \frac{2}t\int_{t/2}^t \rVert \nabla w_t \rVert_{L^2}^2 ds\le_C \frac{\delta}{t^3}.
\end{align}
Noting that  \eqref{thats} and \eqref{energy_decay_bou_2} imply
\begin{align}\label{wanjapleaselet23}
\frac{2}{t}\int_{t/2}^t \rVert w\rVert_{H^2}^2+\rVert u\rVert_{H^1}^2 ds \le_C \frac{\delta}{t^3},
\end{align}
we apply Lemma~\ref{odeep21} to \eqref{rjsowhjxxsd23sd} with $f:=\rVert \nabla w_t\rVert_{L^2}^2$, $g:=C\rVert w_{tt}\rVert_{L^2}^2$ and $h:=C \rVert w_{tt}\rVert_{L^2}^2 + C\left(\rVert w\rVert_{H^2}^2+\rVert u\rVert_{H^1}^2 \right)$, yielding that
\begin{align}\label{rkspw2por1}
\frac{2}{t}\int_{t/2}^t \rVert w_{tt}\rVert_{L^2}^2 ds \le_C \frac{\delta}{t^3}.
\end{align}

\textbf{Estimate for $\rVert w\rVert_{H^4}$.}
From \eqref{leaving2}, \eqref{BouFestimate1} and \eqref{BouFestimate3}, it follows that
\[
\rVert w\rVert_{H^4}\le_C \rVert w\rVert_{H^2} + \rVert u\rVert_{H^1} + (1+\rVert u\rVert_{H^3})\rVert w_{tt}\rVert_{L^2}\le_C \rVert w\rVert_{H^2} + \rVert u\rVert_{H^1} + \rVert w_{tt}\rVert_{L^2}, 
\]
where the last inequality is due to \eqref{size_assu}. Therefore,  \eqref{wanjapleaselet23} and \eqref{rkspw2por1} implies
\begin{align}\label{beawhsd}
\frac{2}{t}\int_{t/2}^t \rVert w\rVert_{H^4}^2 ds \le_C \frac{\delta}{t^3}.
\end{align}

\textbf{Estimate for $\rVert \partial_1\theta \rVert_{L^2}$.}Using the stream function in \eqref{streamv2} and the Gagliardo-Nirenberg interpolation theorem, we observe
 \[
  \rVert \partial_1 \theta\rVert_{L^2}\le \rVert \Delta^2\Phi\rVert_{L^2}\le \rVert \Delta^2\Phi\rVert_{H^2}^{1/2}\rVert \Delta \Phi\rVert_{L^2}^{1/2}\le \rVert \Delta^2\Phi\rVert_{H^2}^{1/2}\rVert \nabla v\rVert_{L^2}^{1/2}\le \rVert \partial_1\Delta\theta\rVert_{L^2}^{1/2}\rVert \nabla v\rVert_{L^2}^{1/2},
 \]
 where the last inequality follows from \eqref{streamv2} and Lemma~\ref{bilap}. Therefore, we have
 \begin{align}\label{almostisd1sd}
 \int_{t/2}^t \rVert \partial_1\theta \rVert_{L^2}^2 ds&\le_C \int_{t/2}^{t}\rVert \partial_1\Delta\theta\rVert_{L^2}\rVert \nabla v\rVert_{L^2}ds\le_C \left(\int_{t/2}^t \rVert \partial_1\Delta\theta\rVert_{L^2}^2 ds\right)^{1/2}\left( \int_{t/2}^t \rVert \nabla v\rVert_{L^2}^{2}ds\right)^{1/2}\nonumber\\
 &\le_C \delta^{1/2}\left( \int_{t/2}^t \rVert \nabla v\rVert_{L^2}^{2}ds\right)^{1/2}\nonumber\\
 &\le_C  \frac{\delta}{t},
 \end{align}
 where the second last inequality is due to \eqref{size_assu} and the last inequality is due to \eqref{energy_decay_bou_2}. Thus, we arrive at
 \begin{align}\label{erika1}
 \frac{2}{t}\int_{t/2}^t \rVert \partial_1\theta\rVert_{L^2}^2ds \le C\frac{\delta}{t^2}.
 \end{align}

\textbf{Estimate for $\rVert u_2\rVert_{H^2}$.}  We can compute, using \eqref{Boussiesq_pertur}, that
 \begin{align}\label{firjskdswsdcomingw1}
 \frac{1}{2}\frac{d}{dt}\rVert \partial_1\theta\rVert_{L^2}^2 = -\int \partial_{11}\theta \theta_t dx = \int \partial_{11}\theta (u_1\partial_1\theta + u_2\partial_2\theta + \partial_2\rho_su_2) dx.
 \end{align}
 Let us decompose the right-hand side as
 \[
\int \partial_{11}\theta u_1\partial_1\theta dx + \int \partial_{11}\theta u_2\partial_2\theta dx +\int \partial_{11}\theta \partial_2\rho_su_2 dx =:I_1+I_2+I_3.
 \]
We start with $I_3$. For $I_3$, we use the stream function $\Psi$ in \eqref{stream22v2} and compute
 \[
 I_3 = -\int \partial_{11}\theta (-\partial_2\rho_s) u_2 dx =- \int \partial_1\Delta^{2}\Psi (-\partial_2\rho_s)u_2 dx=- \int \partial_1\Delta \Psi \Delta(-\partial_2\rho_s u_2) dx,
 \]
 where the integration by parts in the thrid equality is justified by the boundary condition for $u=0$ on $\partial\Omega$ and the incompressible condition, which tell us 
 \begin{align}\label{weksd21}
 u_2 =0,\quad \partial_2u_2 = - \partial_1u_1=0, \text{ on $\partial\Omega$.}
 \end{align}
 Since $\rho_s$ is independent of the horizontal variable $x_1$, we have
 \[
 \Delta(\partial_2\rho_s u_2) =\partial_2\rho_s \Delta u_2 + \partial_{222}\rho_s u_2 + 2\partial_{22}\rho_s\partial_2u_2.
 \]
 Hence,
 \[
 I_3 = - \int \partial_1\Delta \Psi (-\partial_2\rho_s \Delta u_2) dx +  \int \partial_1\Delta \Psi (\partial_{222}\rho_s u_2 + 2\partial_{22}\rho_s\partial_2u_2) dx=: I_{31} + I_{32}.
 \]
 For $I_{31}$, we can further continue, using $w=u-v$ and $v=\nabla^\perp\Psi$, as
 \begin{align}\label{I21sd222}
 I_{31} &=-\int (-\partial_2\rho_s) \partial_1\Delta\Psi \Delta u_2 dx =- \int  (-\partial_2\rho_s) \Delta v_2 \Delta u_2 dx\nonumber\\
 & = -\int (-\partial_2\rho_s) (\Delta u_2)^2 dx - \int(-\partial_2\rho_s) \Delta w_2\Delta u_2dx \le_C -\rVert u_2\rVert_{H^2}^2 + \rVert w\rVert_{H^2}^2,
 \end{align}
 where the last inequality is due to Young's inequality and \eqref{steadtysd2sd}. For $I_{32}$, integrating by parts in $x_1$, we get
 \[
 |I_{32}| = \left|\int \Delta\Psi \left(\partial_{222}\rho_s \partial_1u_2 + 2\partial_{22}\rho_s \partial_{12}u_2 \right)dx\right|\le_C C_\eta\rVert \nabla v\rVert_{L^2}^2 + \eta\rVert \partial_{222}\rho_s\rVert_{L^\infty}^2\rVert u_2\rVert_{H^2}^2, 
 \]
 which holds for any $\eta>0$ due to Young's inequality. Hence, choosing $\eta$ sufficiently small depending on $\rVert\partial_2\rho\rVert_{H^4}$, as assumed in \eqref{steadtysd2sd} and adding the last estimate to \eqref{I21sd222}, we obtain
 \begin{align}\label{I21sd}
 I_3\le_C -\rVert u_2\rVert_{H^2}^2 + \rVert w\rVert_{H^2}^2 + \rVert \nabla v\rVert_{L^2}^2.
 \end{align}
 Now, we move on to $I_2$. Using 
 \[
 \partial_{11}\theta=\partial_1\Delta^2\Psi=\Delta^2v_2,
 \]
 we see that
 \begin{align*}
 I_2 &= \int \Delta^2v_2 u_2 \partial_2\theta dx = \int \Delta^2 u_2 u_2 \partial_2\theta- \int \Delta^2 w_2 u_2\partial_2 \theta dx\\
 & = \int \Delta u_2 \Delta (u_2 \partial_2\theta )dx -\int \Delta^2 w_2 u_2\partial_2\theta dx,
 \end{align*}
 where we performed an integration by parts using \eqref{weksd21}.
 The first integral can be estimated as
 \begin{align*}
 \left|  \int \Delta u_2 \Delta (u_2 \partial_2\theta)dx \right|&\le \rVert \Delta u_2\rVert_{L^2}\rVert \Delta(u_2\partial_2\theta)\rVert_{L^2}\\
 &\le \rVert \Delta u_2\rVert_{L^2}\left( \rVert u_2\rVert_{H^2} \rVert \partial_2 \theta \rVert_{L^\infty}+ \rVert u_2\rVert_{L^\infty}\rVert\partial_2\theta\rVert_{H^2})\right)\\
 &\le_C \rVert u_2\rVert_{H^2}^2 \rVert \partial_2\theta\rVert_{H^2}\\
 &\le_C\rVert u_2\rVert_{H^2}^2\sqrt{\delta_0},
 \end{align*}
 where the last inequality follows from \eqref{size_assu}. For the second integral, Young's inequality gives us
 \[
 \left| \int \Delta^2 w_2 u_2\partial_2\theta dx \right|\le_C \rVert w\rVert_{H^4}^2 +\rVert u_2\rVert_{L^2}^2\rVert \partial_2\theta\rVert_{L^\infty}\le_C  \rVert w\rVert_{H^4}^2 +\rVert u_2\rVert_{L^2}^2\sqrt{\delta_0}
 \]
 Therefore, we get
 \begin{align}\label{sluf2ds21}
 |I_2|\le_C \sqrt{\delta_0}\rVert u_2\rVert_{H^2}^2 + \rVert w\rVert_{H^4}^2.
 \end{align}
 For $I_1$, we compute
 \[
 I_1 =\int \partial_{11}\theta u_1\partial_1\theta dx= -\frac{1}{2}\int (\partial_1\theta)^2 \partial_1u_1 dx = \int (\partial_1\theta)^2 \partial_2 u_2dx,
 \]
 where we used the incompressibility of $u$. Hence, using the Gagliardo-Nirenberg inequality, we get
 \begin{align*}
 |I_1|&\le \rVert \partial_2u_2\rVert_{L^2}\rVert \partial_1\theta\rVert_{L^4}^2 = \rVert \partial_2u_2\rVert_{L^2}\rVert  \Delta^2 \Psi\rVert_{L^4}^2\le \rVert \partial_2u_2\rVert_{L^2} \left(\rVert \Delta \Psi\rVert_{L^2}^{1/2}\rVert \Delta^3 \Psi\rVert_{L^2}^{1/2} \right)^2\\
 &\le\rVert\partial_2u_2\rVert_{L^2} \left( \rVert \nabla v\rVert_{L^2}^{1/2}\rVert \theta\rVert_{H^4}^{1/2}\right)^2\\
 &\le \delta_0 \rVert \partial_2 u_2\rVert_{L^2}^2 + \rVert \nabla v\rVert_{L^2}^2
 \end{align*}
 where the second last inequality is due to $\rVert\Delta^3\Psi\rVert_{L^2} =\rVert \partial_1\Delta \theta\rVert_{L^2}$, which follows from \eqref{streamv2}, and the last inequality is due to \eqref{size_assu} and Young's inequality.
  Thus, Plugging this, \eqref{sluf2ds21} and \eqref{I21sd} into \eqref{firjskdswsdcomingw1}, we arrive at
 \[
 \frac{d}{dt}\rVert \partial_1\theta \rVert_{L^2}^2 \le -\rVert u_2\rVert_{H^2}^2 +\left( \rVert w\rVert_{H^4}^2 +\rVert \nabla v\rVert_{L^2}^2\right).
 \]
 Then, applying Lemma~\ref{odeep21}, with \eqref{erika1}, \eqref{beawhsd}, \eqref{energy_decay_bou_2}, we obtain 
 \[
 \frac{2}{t}\int_{t/2}^t \rVert u_2\rVert_{H^2}^2 ds \le_C \frac{\delta}{t^3}.
 \]
 With \eqref{rkspw2por1}, \eqref{erika1} and \eqref{beawhsd}, we finish the proof.
  \end{proof}
\subsection{Proof of Theorem~\ref{main_Boussinesq}}
Let $\theta_0\in H^{2}_0\cap H^4$ and $u_0\in H^4$ satisfy 
\begin{align}\label{initial_s2}
\rVert \theta_0\rVert_{H^4}+\rVert u_0\rVert_{H^4}\le \epsilon,
\end{align}
 for sufficiently small $\epsilon\ll \sqrt{\delta_0}$, where $\delta_0$ is a constant such that under the assumption \eqref{size_assu}, all the lemmas and the propositions are satisfied. Note that   \eqref{initial_s2} ensures that
 \begin{align}\label{rkjstotlasde}
 E_T(0) = E_P(0) + \rVert u_0\rVert_{L^2}^2 \le_C \rVert \partial_1\theta_0\rVert_{L^2}^2 +\rVert u_0\rVert_{L^2}^2\le_C \epsilon^2,
 \end{align}
 where the second inequality is due to Lemma~\ref{propoos}.
 Also, we claim that
\begin{align}\label{latelast}
\rVert w_t(0)\rVert_{L^2}+\rVert w_{tt}(0)\rVert_{L^2}\le_C \epsilon. 
\end{align}
Indeed, applying Lemma~\ref{leray_pro} to \eqref{evol_w}, we have
  \begin{align}\label{wanjalo}
  \rVert w_t(0)\rVert_{L^2}&\le_C \rVert F(0)\rVert_{L^2} + \rVert \Delta w(0)\rVert_{L^2}\le_C \rVert u_0\rVert_{H^1}+ \rVert w(0)\rVert_{H^2}\nonumber\\
  &\le_C \rVert v(0)\rVert_{H^2}+\rVert u_0\rVert_{H^2}\le_C \rVert \theta_0\rVert_{L^2}+\rVert u_0\rVert_{H^2}\nonumber\\
  &\le_C \epsilon,
  \end{align}
  where the second inequality follows from \eqref{BouFestimate11}.
  Using \eqref{wtt_1} and \eqref{evol_w}, we  also have
  \[
 w_{tt}= -\nabla p_t -F_t + \Delta w_{t} = -\nabla (p_t + \Delta p) - F_t - F + \Delta^2 w. 
  \]
  Thus, applying Lemma~\ref{leray_pro}, we have
  \begin{align*}
  \rVert w_{tt}(0)\rVert_{L^2}&\le_C \rVert F_t(0)\rVert_{L^2}+\rVert F(0)\rVert_{L^2}+ \rVert \Delta^2 w(0)\rVert_{L^2}\\
  &\le_C \rVert u_0\rVert_{H^1}+\rVert w(0)\rVert_{H^2}+\rVert u_0\rVert_{H^3}\rVert w_{tt}(0)\rVert_{L^2}\\
  &\le_C \epsilon +\epsilon\rVert w_{tt}(0)\rVert_{L^2},
  \end{align*}
  where the second inequality follows from \eqref{BouFestimate11} and \eqref{BouFestimate3}. Combining this with \eqref{wanjalo}, we verify the claim \eqref{latelast}.

 We consider the solutions $\theta(t),u(t), w(t)$ to \eqref{Boussiesq_pertur} and \eqref{evol_w}. We claim that 
  \begin{align}\label{proofthensd1}
  \rVert \theta(t)\rVert_{H^4}^2 + \rVert w_t(t)\rVert_{L^2}^2+\rVert w_{tt}(t)\rVert_{L^2}^2 + \int_{0}^t \rVert \partial_1\Delta\theta\rVert_{L^2}^2 + \rVert u\rVert_{H^5}^2+\rVert w\rVert_{H^5}^2dt \le_C \epsilon^2,\text{ for all $t>0$.}
 \end{align}
Once this claim is proved, then together with Lemma~\ref{smallenough} and \eqref{rkjstotlasde}, we see that the assumption~\eqref{size_assu} is satisfied for all time. Especially, we have $\rVert \theta(t)\rVert_{H^4}+\rVert u(t)\rVert_{H^4}\le_C \epsilon$ for all $t>0$, proving the first stability estimate in \eqref{mnum12}. In addition,  we can apply Proposition~\ref{|energyedeasd}, especially \eqref{energy_decay_bou} implies
\[
E_P(t)+E_K(t)\le_C \frac{\epsilon^2}{t^2}.
\]
Thanks to Lemma~\ref{propoos}, we have $\rVert \rho(t)-\rho^*\rVert_{L^2}^2\le_C E_P(t)$. Combining this with the definition of the kinetic energy, we obtain
\[
\rVert \rho(t)-\rho^*\rVert_{L^2}^2 + \rVert u(t)\rVert_{L^2}^2\le_C \frac{\epsilon^2}{t^2},
\]
which proves the second asymptotic stability estimate in \eqref{mnum122}, finishing the proof of the theorem.  In the rest, we will aim to prove \eqref{proofthensd1}.

  To simplify the notations, let us denote
  \[
  f(t):=\rVert \theta(t)\rVert_{H^4}^2 + \rVert w_t(t)\rVert_{L^2}^2+\rVert w_{tt}(t)\rVert_{L^2}^2, \quad g(t):=\rVert \partial_1\Delta\theta(t)\rVert_{L^2}^2 + \rVert u(t)\rVert_{H^5}^2+\rVert w(t)\rVert_{H^5}^2.
  \]
   Towards the proof of \eqref{proofthensd1}, we suppose to the contrary that there exists $T^*>0$ such that
   \begin{align}\label{contransd2sd}
    f(t) + \int_0^{t} g(s)ds \le M\epsilon^2,\text{ for all $t\in [0,T^*]$ and } f(T^*) + \int_0^{T^*} g(s)ds = M\epsilon^2,
   \end{align}
   for some $M\gg 1$ which will be determined in terms of the implicit constant $C>0$, which depends on only $\rho_s$, but not $\epsilon$. Assuming $\epsilon$ is sufficiently small so that $M\epsilon^2\ll \delta_0=\delta_0(\rho_s)$, which guararentees that \eqref{size_assu} is still valid for all $t\in [0,T^*]$.  
   
    We see from Lemma~\ref{smallenough} and \eqref{rkjstotlasde} that
    \begin{align}\label{rememberwanja}
    \rVert u(t)\rVert_{H^4}+\rVert \theta(t)\rVert_{H^4}+\rVert w\rVert_{H^4}\le_C f(t)+E_T(0)\le_C f(t)+ \epsilon^2,\text{ for all $t\in [0,T^*]$.}
    \end{align}
   Recall  $A_1,A_2$ from \eqref{A123def}:
   \begin{align}
   A_1&= \left( \rVert u_2\rVert_{W^{2,\infty}} + \rVert u_2\rVert_{H^3} +  \rVert \Delta^2w \rVert_{L^2}\right)\rVert \Delta^2\theta\rVert_{L^2}\label{A123def3}\\
A_2&=\rVert \partial_1\theta \rVert_{L^2}^2 +\rVert u\rVert_{H^1}^2 +\rVert w_{tt}\rVert_{L^2}^2+ \rVert w\rVert_{H^2}^2.\label{A123def34}
   \end{align}
    The usual Sobolev embedding theorems give us that
    \[
    A_1+A_2\le_C   \rVert u\rVert_{H^4}^2+\rVert \theta\rVert_{H^4}^2+\rVert w\rVert_{H^4}^2 + \rVert w_{tt}\rVert_{L^2}^2\le_C f(t)+ \epsilon^2, \text{ for $t\in [0,T^*]$},
    \]
    where the last inequality is due to \eqref{rememberwanja}.
  Hence, Proposition~\ref{bouenergy_es} implies that 
  \begin{align}\label{differential_ff}
  \frac{d}{dt}f(t)\le -C g(t) + C(A_1+A_2) \le -C g(t) + Cf(t)+C\epsilon^2,\text{ for $t\in [0,T^*]$.}
  \end{align}
 Since $f(0)\le_C\epsilon^2$, which follows from \eqref{initial_s2} and \eqref{latelast},  this differential inequality gives us 
  \begin{align}\label{difpsdiw}
  f(t)\le C\left(\epsilon^2 e^{Ct}+\epsilon^2t\right)\le C\epsilon^2e^{Ct} \text{ for $t\in [0,T^*]$},
  \end{align}
  Therefore,  $ f(T^*) \le_C \epsilon^2 e^{CT^*}$. Integrating the above differential inequality in time, we have
  \begin{align}\label{hhsdkiwisdsd}
  f(T^*) + \int_0^{T^*}g(s)ds\le_C f(0)+\int_0^{T^*}f(t)dt + \epsilon^2 T^*\le_C \epsilon^2\left( 1+ e^{CT^*} +T^*\right)\le\epsilon^2 e^{CT^*}.
  \end{align}
  Thus, \eqref{contransd2sd}  must imply $
  M\epsilon^2 \le_C  \epsilon^2 e^{CT^*}.$
 This inequality tells us that $M\le Ce^{CT^*}$, which gives a lower bound of $T^*$:
  \begin{align}\label{tjjsdowjsdwd}
  T^*\ge C\log M\gg 1.
  \end{align}
  Let us pick 
\begin{align}\label{rkksd21sd}
T^*_M:=\log\log M\gg 1.
\end{align}
From \eqref{tjjsdowjsdwd},  we can assume  $M$ is sufficiently large so that  $T^*_M\ll T^*$. Integrating \eqref{differential_ff} using \eqref{difpsdiw}, we get
\begin{align}\label{nsjdjwdwsad}
f(T^*_M)+\int_0^{T^*_M}g(s)ds\le C\epsilon^2 (\log M)^C + f(0)\le_C \epsilon^2 (\log M)^C.
\end{align}

  Now, we analyze the terms $A_1,A_2$ more carefully.  Using the Gagliardo-Nirenberg interpolation theorem (see \cite[Theorem 5.8]{adams2003sobolev}), we see that 
\begin{align}\label{rjjksd1psdsd}
\rVert u_2\rVert_{W^{2,\infty}} +\rVert u_2\rVert_{H^3}\le_C \rVert u_2\rVert_{H^2}^{2/3}\rVert u_2\rVert_{H^5}^{1/3}.
\end{align}
Thus, Young's inequality gives us
\begin{align*}
\left(\rVert u_2\rVert_{W^{2,\infty}} +\rVert u_2\rVert_{H^3}\right)\rVert \Delta^2\theta\rVert_{L^2}&\le_C \rVert u_2\rVert_{H^2}^{2/3}\rVert u_2\rVert_{H^5}^{1/3}\rVert \Delta^2\theta\rVert_{L^2}\\
&\le_C \eta\rVert u_2\rVert_{H^5}^2 + C_\eta\rVert \Delta^2\theta\rVert_{L^2}^{6/5}\rVert u_2\rVert_{H^2}^{4/5}, \text{ for any $\eta>0$.}
\end{align*}
Plugging this into $A_1$ in \eqref{A123def3} and the estimate for $f(t)$ in \eqref{contransd2sd}, we arrive at
\begin{align}\label{a1sdwsd1}
A_{1}\le_C \eta g(t) +C_\eta \left(M\epsilon^2\right)^{3/5}\rVert u_2\rVert_{H^2}^{4/5} + \left( M\epsilon^2\right)^{1/2}\rVert \Delta^2w\rVert_{L^2}.
\end{align} Hence, choosing $\eta$ small enough and  plugging \eqref{a1sdwsd1} into \eqref{differential_ff}, we get
\begin{align}\label{fianlintegsd}
\frac{d}{dt}f(t)\le -Cg(t) + C\left((M\epsilon^2)^{3/5}\rVert u_2\rVert_{H^2}^{4/5} + (M\epsilon^2)^{1/2}\rVert \Delta^2w(t)\rVert_{L^2}  + A_3(t)\right),\text{ for $t\in [0, T^*]$.}
\end{align}

Towards a contradiction, let us estimate the time integrals of each term in the right-hand side.  We start with the term of $u_2$.
Applying the estimate for $u_2$ in  Proposition~\ref{higherdrhappy} with the assumption \eqref{size_assu} which is valid thanks to \eqref{proofthensd1}, we have that $\frac{2}{t}\int_{t/2}^t \rVert u_2(s)\rVert_{L^2}^2ds\le_C M\epsilon^2 t^{-3}$. In other words, the map $t\mapsto t^{1/4}\rVert u_2(t)\rVert_{H^2}^2$ satisfies
\begin{align*}
\frac{2}{t}\int_{t/2}^t s^{1/4}\rVert u_2(s)\rVert_{H^2}^2 ds\le_C t^{1/4}\frac{M\epsilon^2}{t^3}\le_C \frac{M\epsilon^2}{t^{11/4}}.
\end{align*}
Hence we obtain
\begin{align*}
\int_0^{T^*} s^{1/10}\rVert u_2\rVert_{H^2}^{4/5}  ds = \int_0^{T^*} \left(s^{1/4}\rVert u_2(s)\rVert_{H^2}^2\right)^{2/5}ds \le_C(M\epsilon^2)^{2/5},
\end{align*}
where the last inequality is due to Lemma~\ref{fialmass}.
Thus 
\begin{align}\label{nshdhsd1}
\int_{T^*_M}^{T^*} (M\epsilon^2)^{3/5} \rVert u_2(t)\rVert_{H^2}^{4/5} dt \le (M\epsilon^2)^{3/5}(T^*_M)^{-1/10}\int_{T^*_M}^{T^*} s^{1/10}\rVert u_2(t)\rVert_{H^2}^{4/5} dt\le_C (T^*_M)^{-1/10}M\epsilon^2.
\end{align}
Similarly, the estimate for $\rVert w\rVert_{H^4}^2$ in  Proposition~\ref{higherdrhappy} and the estimates in \eqref{contransd2sd}yield 
\begin{align*}
\frac{2}{t}\int_{t/2}^t s^{1/4}\rVert w(s)\rVert_{H^4}^2 ds\le_C t^{1/4}\left(\frac{2}t\int_{t/2}^t \rVert w(s)\rVert_{H^4}^2)ds \right)\le_C \frac{M\epsilon^2}{t^{11/4}}.
\end{align*}
Hence, applying Lemma~\ref{fialmass}, we obtain
\begin{align*}
\int_0^{T^*}s^{1/8}\rVert w(s)\rVert_{H^4}ds = \int_0^{T^*} \left(s^{1/4}\rVert w(s)\rVert_{H^4}^2 \right)^{1/2}ds \le_C (M\epsilon)^{1/2}
\end{align*}
Thus
\begin{align}\label{nshdhsd2}
\int_{T^*_M}^{T^*} (M\epsilon^2)^{1/2} \rVert w(t)\rVert_{H^4} dt \le (M\epsilon^2)^{1/2}(T^*_M)^{-1/8}\int_{T^*_M}^{T^*} s^{1/8}\rVert u_2(t)\rVert_{H^2} dt\le_C (T^*_M)^{-1/8}M\epsilon^2.
\end{align}
Now, we consider the term $A_2$.
It follows from  Proposition~\ref{higherdrhappy} and \eqref{energy_decay_bou_2} that  $\frac{2}{t}\int_{t/2}^t A_2(s)ds\le_C M\epsilon^2t^{-2}$. Hence, a similar argument as above yields
\begin{align*}
\frac{2}{t}\int_{t/2}^t s^{1/4}A_2(s)ds\le_C t^{1/4}\left(\frac{2}t\int_{t/2}^t A_2(s)ds \right) \le_C\frac{M \epsilon^2}{t^{7/4}}.
\end{align*}
Therefore, again Lemma~\ref{fialmass} gives us $\int_0^{T^*} s^{1/4}A_2(s)ds \le_C M\epsilon^2,$ and  
\begin{align}\label{nshdhsd3}
\int_{T_M^*}^{T^*} A_2(s)ds \le (T^*_M)^{-1/4}\int_{T_M^*}^{T^*}s^{1/4}A_2(s)ds \le_C (T^*_M)^{-1/4}M\epsilon^2.
\end{align}
  Collecting \eqref{nshdhsd1}, \eqref{nshdhsd2} and \eqref{nshdhsd3} and using \eqref{rkksd21sd}, we obtain
  \[
  \int_{T_M^*}^{T^*} (M\epsilon^2)^{3/5}\rVert u_2\rVert_{H^2}^{4/5} + (M\epsilon^2)^{1/2}\rVert \Delta^2w(t)\rVert_{L^2}  + A_2(t) dt\le_C M\epsilon^2 (T_M^*)^{-1/10}\le_C (\log\log M)^{-1/10}M\epsilon^2.
  \]
  Hence integrating \eqref{fianlintegsd} in time over $t\in [T_M^*,T^*]$, we get 
  \[
  f(T^*) + \int_{T_M^*}^{T^*} g(s)ds \le Cf(T_M^*)+C(\log\log M)^{-1/10}M\epsilon^2\le_C \epsilon^2\left(\log M\right)^C+(\log\log M)^{-1/10}M\epsilon^2,
  \] where we used \eqref{nsjdjwdwsad} for the second inequality. 
  Combining this with  the estimate for $\int_0^{T_M^*} g(s)ds$ in  \eqref{nsjdjwdwsad}, we see that
  \[
   f(T^*) + \int_{0}^{T^*} g(s)ds\le_C \epsilon^2\left(\log M\right)^C+(\log\log M)^{-1/10}M\epsilon^2\le_C M\epsilon^2\left( \frac{\left(\log M\right)^C}{M} + \left(\log\log M\right)^{-1/10}\right).
  \]
   Hence  \eqref{contransd2sd} must imply
  \[
  M\epsilon^2  \le C \epsilon^2M\left(\frac{(\log M)^C}{M}+ (\log\log M)^{-1/10}\right),
  \]
  in other words,
  \[
  1\le C \left((\log\log M)^{-1/10} + \frac{(\log M)^C}{M} \right).
  \]
  Since $C$ is a constant that depends on only the steady state $\rho_s$, we can choose $M$ large enough so that this inequality leads to a contradiction. This proves the claim \eqref{proofthensd1} and finishes the proof. 
    \begin{appendix}
  \section{ODE estimates}\label{ode11sx}

In the appendix, we collect elementary ODE estimates.   
\begin{lemma}\label{odelem}
Let $f,g,\alpha:[0,T]\mapsto \mathbb{R}^+$ be smooth functions for some $T>0$. Suppose it holds that
\begin{align}\label{oporsd21}
\frac{d}{dt} (f(t)+g(t))\le -\left( f(t)^2 \alpha(t)^{-1} + g(t)\right),\text{ for $t\in [0,T]$.}
\end{align}
Then we have
\[
f(T)+g(T)\le C\frac{A}{T^2}, \text{ where $A:=\max\left\{ \int_0^T\alpha(s)ds, f(0)+g(0)\right\}$}.
\]
\end{lemma}
\begin{proof}
Let us decompose the interval $[0,T]$ as
\[
[0,T] =: I\cup J,\text{ where } I:=\left\{ t\in [0,T]: g(t)\le \alpha(t)\right\},\quad  J:=\left\{ t\in [0,T]: g(t)> \alpha(t)\right\}.
\]
Since $f,g,\alpha$ are smooth functions, we can decompose each of $I_1,I_2$ into a union of disjoint intervals:
\[
I=\cup_{i\in\mathbb{N}\cup\left\{ 0\right\}}[t_{2i},t_{2i+1}]=:\cup_{i\in\mathbb{N}\cup\left\{ 0\right\}} I_i,\quad J=\cup_{i\in \mathbb{N}\cup\left\{ 0\right\}}[t_{2i+1},t_{2i+2}]=:\cup_{i\in\mathbb{N}\cup\left\{ 0\right\}} J_i,\text{ for $t_i\le t_{i+1}$.}
\]
Let us consider two cases, either $|I|\ge |J|$ or $|I|\le |J|$. We suppose first $|I|\ge |J|$. On each $I_i$, we have
\[
g(t)^2\alpha^{-1} =g(t)\times (g(t)\alpha(t)^{-1})\le g(t),\text{ for $t\in I$.} 
\]
Hence, the differential inequality \eqref{oporsd21} reads as
\[
\frac{d}{dt} (f(t)+g(t)) \le -(f(t)^2\alpha(t)^{-1} +g(t)^2 \alpha(t)^{-1}) = -C \left(f(t)+g(t) \right)^2 \alpha(t)^{-1},\text{ for $t\in I$.}
\] 
Denoting $h(t):=f(t)+g(t)$, we have
\[
\frac{d}{dt}h(t)\le -C h(t)^2 \alpha^{-1}(t), \text{ for $t\in I$.}
\]
Integrating the both sides over $t\in I_i$, we get
\[
\frac{1}{h(t_{2i+1})} - \frac{1}{h(t_{2i})}\ge C\int_{t_{2i}}^{t_{2i+1}}\alpha(s)^{-1}ds.
\]
Summing the both sides over $i$ and using that $\frac{1}{h(t_{2i+2})}-\frac{1}{h(t_{2i+1})}\ge0$ for all $i$ due to the monotonicity of $t\mapsto h$, we get, denoting $t_{I,\infty}:=\sup_{t\in I}t\le T$, 
\[
\frac{1}{h({t_{I,\infty}})} - \frac{1}{h(t_{0})}\ge  \sum_{i=0}^\infty \int_{I_i}\alpha(s)^{-1}ds = \int_{I}\alpha(s)^{-1}ds \ge \frac{|I|^2}{\int_{I}\alpha(s)ds},
\]
where the last inequality follows from Jensen's inequality.  Again since $h$ is nonnegative and monotone decreasing, the left-hand side satisfies
\[
\frac{1}{h({t_{I,\infty}})} - \frac{1}{h(t_{0})} \le \frac{1}{h(t_{I,\infty})}\le \frac{1}{h(T)}.
\]
On the other hand, the assumption that $|I|\ge |J|$ implies $|I| \ge \frac{1}{2}T$, therefore, 
\[
\frac{|I|^2}{\int_{I}\alpha(s)ds} \ge \frac{CT^2}{\int_{I}\alpha(s)ds}\ge \frac{CT^2}{\int_{0}^T\alpha(s)ds}.
\]
Thus we arrive at $\frac{1}{h(T)}\ge \frac{CT^2}{\int_0^T \alpha(s)ds}$, that is,
\begin{align}\label{rkawksomething}
{h(T)}\le \frac{\int_0^T \alpha(s)ds}{CT^2},\text{ if $|I|\ge |J|$.}
\end{align}
Now, let us consider the case where $|I|\le |J|$. On the interval $J_i$,  Young's inequality yields
\[
f(t) = f(t)\alpha(t)^{-1/2}\alpha^{1/2}\le \frac{1}{2}f(t)^2\alpha^{-1}(t) + \frac{1}{2}\alpha(s)\le \frac{1}{2 }f(t)^2\alpha(t)^{-1} +\frac{1}{4}g(t),\text{ for $t\in J$.}
\]
Therefore,
\[
\frac{1}{2 }f(t)^2\alpha(t)^{-1} +g(t) \ge_Cf(t) + g(t),\text{ for $t\in J$.}
\]With this, the differential inequality \eqref{oporsd21} reads as
\[
\frac{d}{dt}h(t) \le -Ch(t),\text{ for $t\in J$.}
\]
Solving this, we obtain $h(t)\le h(t_{2i+1})e^{-C(t-t_{2i+1})}$ for $t\in J_i$. In other words, we get 
\[
\frac{h(t_{2i+2})}{h(t_{2i+1})}\le e^{-C|J_i|},\text{ for all $i$.}
\]
Since $t\mapsto h(t)$ is monotone decreasing, we have $\frac{h(t_{2i})}{h(t_{2i-1})}\le 1$ for $i\ge 1$, hence, denoting $t_{J,\infty}:=\sup_{t\in J}\le T$. we obtain
\[
\frac{h(T)}{h(0)}\le \frac{h(T)}{h({t_1})} \le \Pi_{i=0}^\infty \frac{h({t_{2i+2}})}{h(t_{2i+1})} \le \Pi_{i=0}^\infty e^{-C|J_i|}\le e^{-C|J|}\le e^{-CT},
\]
where the last inequality follows from the assumption that $|J| \ge |I|$, thus $|J|\ge \frac{1}{2}|T|$. Thus we arrive at 
\[
h(T)\le h(0)e^{-CT}\le \frac{Ch(0)}{T^2},\text{ if $|J|\ge |I|$.}
\]
Combining this with \eqref{rkawksomething}, we obtain the desired result.
\end{proof}

\begin{lemma}\label{odeep21}
Let $f,g,h:[0,T]\mapsto \mathbb{R}^+$ be smooth functions for some $T>0$. Suppose it holds that
\[
\frac{d}{dt}f(t)\le -g(t) +h(t),\text{ for $t\in [0,T]$},
\]
and $f,h$ satisfy  decay conditions:
\[
\frac{2}{t}\int_{t/2}^{t}f(s)ds \le \frac{A}{t^{n-1}},\quad \frac{2}{t}\int_{t/2}^{t}h(s)ds \le \frac{A}{t^{n}},\text{ for some $n\in\mathbb{N},A>0$ and for all $t\in [0,T]$.}
\]
Then,
\[
\frac{2}{t}\int_{t/2}^{t}g(s)ds \le \frac{A}{t^n},\text{ for all $t\in [0,T]$.}
\]
\end{lemma}
\begin{proof}
Integrating the differential inequality in time over $[s,t]$ for $0<s<t<T$, we get $f(t)-f(s) +\int_{s}^t g(u)du \le \int_s^t h(u)du$, thus
\[
\int_{s}^t g(u)du \le \int_s^t h(u)du + f(s).
\] 
Integrating the both sides in $s$ over $[t/4,t/2]$, we see that the left-hand side can be estimated as
\begin{align*}
\int_{t/4}^{t/2}\int_{s}^t g(u)duds &= \int_{t/4}^{t/2}\int_0^t g(u)1_{u\ge s}duds = \int_0^t g(u)\int_{t/4}^{t/2}1_{u\ge s}duds \ge \int_{0}^t g(u)\int_{t/4}^{t/2}1_{u\ge t/2}dsdu \\
&= \int_{t/2}^t g(u)\frac{t}{4}du.
\end{align*}
For the right-hand side estimate, we compute
\[
\int_{t/4}^{t/2} \int_s^{t}h(u) duds \le \int_{t/4}^{t/2}\int_{t/4}^t h(u)duds = \frac{t}4\left(\int_{t/2}^{t}h(u)du +\int_{t/4}^{t/2}h(u)du \right)\le Ct\frac{A}{t^{n-1}}\le \frac{CA}{t^{n-2}},
\]
where we used the assumption on $h$ for the second last inequality. Also the following estimate is trivial:
\[
\int_{t/4}^{t/2}f(s)ds \le \frac{CA}{t^{n-2}}.
\]
Putting them together, we arrive at
\[
t\int_{t/2}^t g(u)du \le \frac{CA}{t^{n-2}}. 
\]
Dividing the both sides by $t^2$, we obtain the desired result.
\end{proof}

\begin{lemma}\label{odee2p21}
Let $f,g:[0,T]\mapsto \mathbb{R}^+$ be smooth functions for some $T>0$. Suppose it holds that
\[
\frac{d}{dt}f(t)\le -f(t) +g(t),\text{ for $t\in [0,T]$},
\]
and $g$ satisfies a decay condition:
\[
\frac{2}{t}\int_{t/2}^{t}g(s)ds \le \frac{A}{t^{n}},\text{ for some $n\in\mathbb{N},A>0$ and for all $t\in [0,T]$.}
\]
Then,
\[
\frac{2}{t}\int_{t/2}^{t}f(s)ds \le \frac{A+B}{t^n},\text{ for all $t\in [0,T]$, where $B:=f(0)+\int_0^T g(s)ds$.}
\]
\end{lemma}
\begin{proof}
Let $F$ be a solution to
\[
\frac{d}{dt}F(t)=-F(t)+g(t),\ F(0)=f(0).
\]
The comparison principle tells us  that
\[
f(t)\le F(t),\text{ for all $t\in [0,T]$.}
\]
On the other hand, the method of integrating factor tells us
\begin{align*}
F(t)&= e^{-t}F(0) + \int_0^t g(s)e^{s-t}ds= e^{-t}f(0) + \int_0^{t/2}g(s)e^{s-t}ds + \int_{t/2}^{t}g(s)e^{s-t}ds \\
&\le e^{-t}f(0) + e^{-t/2}\int_0^{t/2}g(s)ds + \int_{t/2}^{t}g(s)ds \\
& \le e^{-t/2}\left(f(0) + \int_0^{t/2}g(s)ds \right) + \int_{t/2}^t g(s)ds.
\end{align*}
Denoting
\[
B:=f(0)+\int_0^t g(s)ds,
\]
and using the decay assumption on $g$, we obtain
\[
F(t)\le e^{-t/2}B +\frac{A}{t^{n-1}}\le \frac{A+B}{t^{n-1}}.
\]
Hence, $f(t)\le \frac{A+B}{t^{n-1}}$, and thus $\frac{2}t\int_{t/2}^t f(s)ds \le_C \frac{A+B}{t^{n-1}}$. Applying Lemma~\ref{odeep21}, we arrive at
\[
\frac{2}t\int_{t/2}^t f(s)ds \le \frac{A+B}{t^{n}},\text{ for all $t\in [0,T]$.}
\]
\end{proof}

Lastly, a proof of the next lemma can be found in \cite[Lemma 2.3]{park2024stability}.
\begin{lemma}\cite[Lemma 2.3]{park2024stability}\label{fialmass}
Let $T> 2$ and $n> 1$. Let a nonnegative function $f:[0,T]\mapsto \mathbb{R}^+$ satisfies
\[
\frac{2}{t}\int_{t/2}^{t} f(s)ds\le \frac{E}{t^n},\text{ for some $E>0$, for all $t\in [2,T]$}.
\]
Then, for $\alpha\in (1/n,1]$, we have 
\[
\int_1^T f(t)^{\alpha}ds \le C_{\alpha,n}E^{\alpha},
\]
where $C_{\alpha,n}>0$ does not depend on $T$.
\end{lemma}

  \end{appendix}

  \bibliographystyle{abbrv}
\bibliography{references}

\end{document}